\documentclass[AMS,Times1COL]{WileyNJDv5} 

\usepackage{boldfonts}
\usepackage{mydef}
\usepackage{bbm}

\newcommand{\bmom}{\mathbf{m}}
\newcommand{\inner}[2]{\langle #1, #2 \rangle}        

\articletype{Article Type}%

\received{Date Month Year}
\revised{Date Month Year}
\accepted{Date Month Year}
\journal{Numerical Methods for Partial Differential Equations}
\volume{00}
\copyyear{2023}
\startpage{1}

\begin{document}

\title{Artificial compressibility method for the incompressible Navier-Stokes equations with variable density}

\author[1]{Loic Cappanera}

\author[1]{Salvatore Giordano}

\authormark{Cappanera \textsc{et al.}}
\titlemark{Artificial compression methods for incompressible multiphase flows}

\address{\orgdiv{Department of Mathematics}, \orgname{University of Houston}, \orgaddress{Houston, \state{Texas}, \country{USA}}}



\corres{Corresponding author: Loic Cappanera. \email{lmcappan@central.uh.edu}}


\fundingInfo{the National Science Foundation NSF (L. Cappanera, grant
number DMS-2208046)}

\abstract[Abstract]{
We introduce a novel artificial compressibility technique to  approximate the incompressible Navier-Stokes equations with variable fluid properties such as density and dynamical viscosity. The proposed scheme used the couple pressure and momentum, equal to the density times the velocity, as primary unknowns. It also involves an adequate treatment of the diffusive operator such that treating the nonlinear convective term explicitly leads to a scheme with time independent stiffness matrices that is suitable for pseudo-spectral methods. The stability and temporal convergence of the semi-implicit version of the scheme is established under the hypothesis that the density is approximated with a method that conserves the minimum-maximum principle. Numerical illustrations confirm that both the semi-implicit and explicit scheme are stable and converge with order one under classic CFL condition. Moreover, the proposed scheme is shown to perform better than a momentum based pressure projection method, previously introduced by one of the authors, on setups involving gravitational waves and immiscible multi-fluids in a cylinder. 
}

\keywords{multiphase flows, incompressible flows, artificial compressibility technique, finite element method, pseudo-spectral method}


\maketitle

\section{Introduction}
\label{sec:intro}

The numerical approximation of incompressible flows with variable density and viscosity plays a crucial role in many engineering and physics applications in fluid mechanics. For instance, such setups occur in nature in the stratified layers of the Earth's mantle and also in the stratified layers of the ocean and their interaction with the atmosphere. Similarly, variable density flows and their interaction with magnetic field and heat convection play an important part in the design of liquid metal batteries and aluminum Hall-H\'{e}roult production cell.
All of the above problems are governed by the incompressible Navier-Stokes equations coupled with a transport equation that governs the dynamics of the density, see equation \eqref{eq:Navier_Stokes}. The case of constant density flows has been the topic of a broad literature these past sixty years. The main challenge consists of enforcing the incompressibility constraints, which is addressed using three main class of methods. First, one can consider the coupled problem which leads to solving a linear system with a saddle-point structure, see \cite{cahouet1988some,fortin1983augmented,benzi2006augmented,olshanskii2022recycling}. The other widely popular approach consists of using projection type techniques that were originally introduced by Chorin \cite{Chor68} and Temam \cite{tema69} which have led to many improvements since, see \cite{MR2250931} and references therein for a review of projection methods. Eventually, another approach consists of the artificial compressibility techniques that were introduced in \cite{chorin1967numerical,temam1968methode, ladyzhenskaya1969mathematical}. This last approach presents a better computational complexity compared to projection methods as it does not require solving Poisson problem to update the pressure, which scales in $h^{-2}$ with $h$ being the mesh size. However, its limited accuracy in time to the order one, until recent developments \cite{guermond2015high} that require solving $n$ parabolic equations to get an artificial compressibility method of order $n$, have limited its popularity.

In the context of immiscible incompressible multiphase flows, i.e. flows with variable density, another difficulty consists of tracking the evolution of the interface between the fluids to reconstruct accurately the fluids properties such as its density. It led to the development of methods such as the level set technique \cite{osher1988fronts,osher1993level}, 
and the phase field methods \cite{anderson1998diffuse,liu2003phase,chiu2011conservative}.
Additionally, approximating efficiently the Navier-Stokes equations also becomes more challenging as the Navier-Stokes system now involves a partial differential equation with variable coefficients in the time derivative operator, and also in the diffusive operator when the viscosity is variable.
The most popular methods of approximation use the velocity as primary unknown and projection-correction methods \cite{badalassi2003computation,GLLN09,chiu2011conservative,
dong2012time}. 
Recent works also focus on the development of artificial compressibility methods that have introduced a high order method in time \cite{lundgren2023high}, an entropy stable method for discontinuous finite element \cite{manzanero2020entropy}, and an unconditional stable method based on SAV technology \cite{zhang2023mass}. One of the main drawback of the above methods is that the resulting linear algebra requires to reassemble the mass matrix, that takes the form $\rho \partial_t \bu$, and the diffusive operator at each time iteration which is either computationally expensive for high order finite element discretization or cannot be made implicit for spectral and pseudo-spectral methods. To develop algorithms that are suitable for spectral methods, and that uses a time independent linear algebra, one of the authors proposed and studied a novel projection method that uses the momentum as primary unknown in \cite{cappanera_2018,cappanera_vu_2024}. This method is shown to be stable and accurate on various setups like rising bubble and rotating multifluids in cylinders. Unfortunately, it shows some robustness limitations when approximating complex multi-fluid with strong interface deformation like the one occurring in aluminum production cell \cite{nore2021feasibility,herreman2023stability}. It led the authors to use the idea of \cite{cappanera_2018} to develop a novel method that uses an artificial compressiblity technique to enforce the incompressibility constraints. The resulting method has been shown to be robust and accurate for aluminum production cell setup in \cite{nore2021feasibility}, however, the analysis of its stability and temporal convergence has not been establish yet. Thus, we aim to fill this gap with this paper.

The paper is organized as follows. 
Section \ref{sec:problem_description} introduces the governing equations, some notations, and the main idea of the artificial compressibility technique we study in this paper. In section \ref{sec:stab_algo}, we establish the stability of the proposed method under the assumption that the density is approximated with a method that satisfies the minimum-maximum principle. Then, under similar assumptions, we establish the time error analysis of the proposed algorithm in section \ref{sec:time_error_analysis}. The section \ref{sec:num_results} reports numerical results obtained with two different codes that either use finite element or pseudo-spectral methods for the spatial discretization. Eventually, we give concluding remarks in the section \ref{sec:conclusion}.

\section{Problem description and artificial compressibility method} 
\label{sec:problem_description}

In this section, we first describe the Navier-Stokes system for problems with variable density and viscosity and some classic notations. Then, focusing on the approximation of the couple velocity-pressure, we describe the key ingredients that allow us to derive an artificial compressibility method that is suitable for spectral methods and that also has the advantage of involving a time-independent stiffness matrix.

\subsection{The incompressible Navier-Stokes equations with variable density and viscosity}
\label{sec:Navier_eq}
We consider an open bounded domain $\Omega$ in $\Real^d$, with $d=2,3$, and a time interval $[0,T]$. The incompressible Navier-Stokes system on $\Omega \times [0,T]$ is made up of the mass conservation, momentum conservation, and incompressibility equations. The system can be expressed as follows:
\begin{subequations}\label{eq:Navier_Stokes}
\begin{align}
 &\partial_t\rho   + \DIV\bmom =0, \label{eq:mass_equation}\\ 
 &\partial_t\bmom   + \DIV( \bmom{\otimes}\bu)  
-  2 \DIV (\eta(\rho) \varepsilon(\bu)) + \GRAD p 
 =  \bef, 
\label{eq:momentum_equation}\\
&\DIV \bu   =    0,  \label{eq:incompressibility_equation} 
\end{align}
\end{subequations}
where $\bu$ is the velocity, $\bmom=:\rho\bu$ is the momentum, $\rho$ is the density,
$\eta$ is the dynamical viscosity, $p$ is the pressure, $\bef$ is a source term and
$\varepsilon(\bu):=\GRAD^s \bu = \frac12 (\GRAD \bu +(\GRAD \bu)\tr)$
is the strain rate tensor. The above system is supplemented with initial condition and boundary conditions. To avoid nonrelevant difficulties, we assume that homogeneous Dirichlet boundary conditions are enforced on the velocity. We also note that we assume that the dynamical viscosity, $\eta$, is a Lipschitz function of the density $\rho$. 



\subsection{Notations and preliminaries}
\label{sec:notations_preliminaries}
 
To approximate in time the system of equations \eqref{eq:Navier_Stokes}, 
we introduce a time step denoted by $\tau$ and set $t_n=n\tau$ for all 
integers $n\geq0$. 
For any time dependent function $\phi$, we set $\phi^n=\phi(t_n)$.
To simplify notations, we denote by $\phi^\tau$ the discrete time
sequence $(\phi^n)_{n\geq0}$ and we introduce the time-increment 
operator $\delta$ defined as follows:
\begin{equation}\label{eq:def_delta_op}
\delta \phi^n = \phi^n - \phi^{n-1}.
\end{equation}
We also let $N$ be the integer such that $N\tau = T$ the final time.

In the rest of the paper, vector variables are written in bold.
We use the standard notation $\| \cdot \|_{L^2}$, respectively 
$\| \cdot \|_{L^\infty}$, to represent the $L^2$ norm, respectively 
$L^\infty$ norm, of a scalar. 
The $\bL^2$ and $\bL^\infty$ norms of a vector or a matrix are denoted by $\| \cdot \|_{\bL^2}$ and  $\| \cdot \|_{\bL^\infty}$.
For any vector $\bu$, we denote by $\|\bu\|_{l^2}$ the euclidean norm of $\bu$ and $\|\varepsilon(\bu)\|_{l^2}$ the induced matrix norm of $\varepsilon(\bu)$.

We conclude this sections by reminding some well known inequalities that will be used in the following sections. There exists a constant $c_1$ such that for all vector $\bu$ in $\bH^1(\Omega)$ such that $\bu_{|\partial \Omega}=0$, the Poincar\'e inequality holds: 
\begin{equation}
\label{eq:prelim_Poincare}
\| \bu^n \|_{\bL^2} \leq c_1 \|\GRAD \bu^n \|_{\bL^2}.
\end{equation}
The first Korn inequality yields the existence of a constant $c_2$ such that
for all $\bu\in\bH^1(\Omega)$:
\begin{equation}
\label{eq:prelim_Korn}
\| \GRAD \bu^n \|_{\bL^2} \leq c_2 \|\varepsilon(\bu^n) \|_{\bL^2}.
\end{equation}

Similarly, one can show that there exists a constant $c_3$ such that for all 
$\bu\in\bH^1(\Omega)$ we have:
\begin{equation}
\label{eq:prelim_div}
\| \sqrt{\eta} \DIV \bu^n \|_{L^2} 
\leq 
c_3 \| \sqrt{\eta} \varepsilon(\bu^n) \|_{\bL^2},
\end{equation}
where $\eta$, the dynamical viscosity, can be any positive function bounded from above and away from zero. We will also use the following polarization identity, that holds for any real $(a,b)$ without referring to it:
$$2(a-b)a = a^2 - b^2 + (a-b)^2.$$

\subsection{Artificial compressibility method based on momentum approximation}
\label{sec:art_comp_idea}
We now aim to introduce a numerical method for the system of equations
\eqref{eq:momentum_equation}-\eqref{eq:incompressibility_equation}
that is suitable for spectral methods and that uses a time-independent stiffness matrix.
The development of such a method faces three main challenges, described below, that we propose to overcome by using the momentum as a primary unknown combined to a method based on artificial compressibility technique for the Navier-Stokes equations.

The first challenge consists of approximating the mass matrix $\partial_t(\rho \bu)$, or $\rho \partial_t \bu$, which cannot be made implicit when using spectral methods due to the product between the density with the velocity or its time derivative. 
As shown by the authors in \cite{cappanera_2018}, this issue can be overcome by using the momentum $\bmom:=\rho \bu$ as primary unknown for the momentum equation \eqref{eq:momentum_equation}. 
The mass matrix then becomes $\partial_t \bmom$ which can be approximated straightforwardly by any spectral or finite element methods. We note that such strategy is commonly used in compressible fluid mechanics.

The second main challenge is related to the treatment of the diffusive term $-\DIV(\eta \varepsilon(\bu))$. 
As the dynamical viscosity $\eta$ is time and space dependent, and that $\bu=\frac{1}{\rho} \bmom$, this term can not be made implicit for spectral methods in the context of parallel simulations. 
This difficulty can be answered by adapting techniques introduced in
\cite{gottlieb1977numerical,badalassi2003computation,dong2012time,cappanera_2018} 
that leads to rewrite the diffusive term 
as follows:
$- \DIV(\bar{\nu} \varepsilon(\bmom))  
+ \DIV(\bar{\nu} \varepsilon(\bmom)) -\eta \varepsilon(\bu)))$
with $\bar{\nu}$ a constant larger than the kinematic viscosity $\nu:=\eta/\rho$. The first term can then be treated implicitly while the correction is made explicit. We note that, for pseudo spectral methods, $\bar{\nu}$ can also be defined as a function that only depends of the space coordinates that are not discretized with spectral elements.

Eventually, the last difficulty is to enforce the incompressibility condition
\eqref{eq:incompressibility_equation}.
Most of the algorithms for incompressible multiphase flows enforce this condition
 by using a projection-correction method 
\cite{badalassi2003computation,cappanera_2018,chiu2011conservative,
dong2012time,GLLN09}
which leads to solve a Poisson problem
to update the pressure. To avoid solving a Poisson problem, whose condition
number scales in $\calO(h^{-2})$ with $h$ the mesh size, we propose to enforce
the incompressiblity condition by developing an algorithm based on artificial
compressibility methods that were originally introduced in \cite{chorin1967numerical, temam1968methode, ladyzhenskaya1969mathematical}. 
We remind that the main idea consists of modifying equation
\eqref{eq:incompressibility_equation} as follows:
\begin{equation}
 \epsilon \partial_t p + \DIV \bu = 0,
\end{equation}
where $\epsilon$ is a tunable constant. Introducing $\lambda:=\tau/\epsilon$ and discretizing in time the above equation reads:
\begin{equation}
 p^{n+1} - p^n + \lambda \DIV \bu^{n+1} = 0.
\end{equation}
Thus, the gradient of the pressure at time $t_{n+1}$ can be approximated using the gradient of the pressure at time $t_n$ and the gradient of the velocity's divergence.
The time discretization of the momentum equation \eqref{eq:momentum_equation} then leads to the approximation of a term of the form $ -\lambda \GRAD(\DIV \bu^{n+1})$. 
Since we aim to use the momentum as
primary unknown, this term is treated in a similar way than the diffusion term. That is, we rewrite the term as follows
$-\lambda \GRAD( \frac{1}{\underline{\rho}}\DIV \bmom) 
+ \lambda \GRAD( \frac{1}{\underline{\rho}}\DIV \bmom - \DIV \bu)$
with $\underline{\rho}$ a positive constant smaller than the density $\rho$. As shown in the following sections, it yields a stable and convergent scheme when the first term is made implicit while the correction is made explicit.

Disregarding the approximation of the density and the nonlinear term $\DIV(\bmom\otimes\bu)$, the algorithm we propose to approximate the system \eqref{eq:momentum_equation}-\eqref{eq:incompressibility_equation} with the couple momentum-pressure reads:
\begin{equation}
\frac{\delta \bmom^{n+1}}{\tau}
- 2 \bar\nu \DIV ( \varepsilon ( \bmom^{n+1} - \bmom^{*,n+1}))
- \Bar{\lambda} \GRAD ( \DIV (\bmom^{n+1} - \bmom^{*,n+1}))
=
 \bef^{n+1}
 - \GRAD p^n 
+ 2\DIV(\eta^{n+1} \varepsilon(\bu^n))
+ \lambda \GRAD (\DIV \bu^n),
\end{equation}

\begin{equation}
p^{n+1} = p^n -  \lambda \DIV\bu^{n+1},
\end{equation}
where $\bmom^{*,n+1}:=\rho^{n+1} \bu^n$ and $\Bar{\lambda}>\frac{\lambda}{\underline{\rho}}$. We note that this time stepping is
suitable for spectral method as nonlinearity are all made explicit. Moreover, the stiffness
matrix is indeed time independent as the variables $\bar{\nu}, \underline{\rho}$ and 
$\lambda$ are constants. The use of an artificial compressibility technique also allows us
to avoid solving a Poisson problem to update the pressure. All these features
combined makes this algorithm efficient for parallel and large scale computing with 
spectral and finite element methods.
\section{Stability of a semi-implicit algorithm}
\label{sec:stab_algo}
We introduce a time discretization of the Navier-Stokes equations \eqref{eq:Navier_Stokes}. We assume that the density is approximated in a way that satisfies the minimum-maximum principle and focus on establishing the stability of the resulting momentum-pressure algorithm. Unlike the model scheme introduced in the previous section, we reintroduce the nonlinear term 
$\DIV(\bmom\otimes \bu)$ which is treated semi-implicitly for analysis purposes. 
The resulting algorithm is shown to be stable under some conditions on the time step $\tau$ and the sequence $\rho^\tau$.


\subsection{Approximation of Density and maximum principle hypothesis.}
\label{sec:approximation_density}

As the main novelty of this work is how the momentum equation and incompressibility constraints \eqref{eq:momentum_equation}-\eqref{eq:incompressibility_equation} are approximated, we do not focus on the mass conservation equation. 
Meaning, we assume that the approximated density satisfies a scheme of the form:
\begin{equation}
\label{eq:approximation_density}
\frac{\delta\rho^{n+1}}{\tau}+ \bu^n\cdot\nabla\rho^n=R_\rho^{n+1},
\end{equation}
where the residual term $R_\rho$ is assumed to be consistent. 

We further assume that the resulting sequence $\rho^\tau$ satisfies the maximum-minimum, meaning that it satisfies
\begin{equation}\label{hyp:min_max_principle}
\rho_{\min} \leq \rho^n \leq \rho_{\max},  \qquad \forall n\geq 0,
\end{equation}
where $\rho_\text{min}=\min_{\bx\in\Omega} \rho^0(\bx) $
and $\rho_\text{max}=\max_{\bx\in\Omega} \rho^0(\bx) $.

We note that many bound preserving methods that satisfies the above properties are available in the literature. For instance, one can use the flux corrected transport technology \cite{boris1973flux,zalesak1979fully} or
total variation diminishing limiters \cite{harten1984class,harten1997high}. One can also use the more recent algebraic flux correction schemes \cite{kuzmin2012algebraic,badia2017monotonicity,barrenechea2017algebraic,guermond2017conservative,kuzmin2020monolithic}.
We refer to \cite{kuzmin2012flux} for a review on flux-corrected transport methods.

To conclude, we recall that the dynamical viscosity is assumed to be a Lipschitz function of the density. As a consequence, there exists a positive constant $\alpha$ such that for all $(\rho_1,\rho_2) \in [\rho_{\min},\rho_{\max}]$, we have:
\begin{equation}
\label{hyp:Lipschitz_continuity}
|\eta(\rho_1)-\eta(\rho_2)|\leq \alpha |\rho_1-\rho_2|.
\end{equation}
Note that the maximum-minimum principle assumption on $\rho^\tau$, see \eqref{hyp:min_max_principle}, implies that $\eta^\tau$ also satisfies the maximum-minimum principle, that is 
\begin{equation}
\label{hyp:min_max_principle_eta}
0
< \eta_{min}
\leq \eta^n  = \eta(\rho^n) 
\leq \eta_{max}.
\end{equation}

\subsection{Time marching algorithm}
\label{sec:time_algo}
After initializing the unknowns $(\rho^0,\bmom^0,p^0)$ with the initial conditions associated to
the problem \eqref{eq:Navier_Stokes}, the proposed time marching algorithm reads as follows. Let $n\in \mathbb{N}$ and assume that $(\rho^n,\bmom^n,\bu^n,p^n)$ are known. Then we follow the sequential scheme below to update the density, momentum, pressure, and velocity.

Step 1. Find $\rho^{n+1}$ the solution of:
\begin{equation}\label{eq:algo_step_rho}
\frac{\delta \rho^{n+1}}{\tau} + \bu^n \cdot \GRAD \rho^n  = R_\rho^{n+1}.
\end{equation}

Step 2. Compute the momentum $\bmom^{n+1}$ by solving:
\begin{multline}\label{eq:algo_step_mom}
\frac{\delta \bmom^{n+1}}{\tau}
- 2\bar\nu\DIV ( \varepsilon ( \bmom^{n+1} - \bmom^{*,n}))
- 2\DIV(\eta^{n+1} \varepsilon(\bu^n))
- \Bar{\lambda} \GRAD ( \DIV (\bmom^{n+1} - \bmom^{*,n}))
 - \lambda \GRAD (\DIV \bu^n)
+  \GRAD p^n 
\\
+ \DIV(\rho^n\bu^{n+1}{\otimes} \bu^n) 
- \frac{1}{2} \rho^n\bu^{n+1}\DIV \bu^{n}
- \frac{1}{2} \bu^{n+1} R_\rho^{n+1}
=
\bef^{n+1},
\end{multline}
with $\bmom^{*,n}=\rho^{n+1} \bu^n$. We note that the above step involves the quantity $\bu^{n+1}$, we chose this notation for simplicity. However, in practice it is replaced by $\frac{1}{\rho^{n+1}}\bmom^{n+1}$ because we use the momentum as primary unknown.

Step 3. Compute the velocity $\bu^{n+1}$ using:
\begin{equation}\label{eq:algo_step_vel}
\bu^{n+1} = \frac{1}{\rho^{n+1}} \bmom^{n+1}
\end{equation}

Step 4. Update the pressure $p^{n+1}$ as follows:
\begin{equation}\label{eq:algo_step_pre}
p^{n+1} = p^n -  \lambda \DIV \bu^{n+1}.
\end{equation}

As shown in the following, to ensure the 
stability of the algorithm the quantities $\bar{\nu}$ and $\bar{\lambda}$ are defined as follows:
\begin{equation}\label{eq:def_nu_rho_bar}
\bar{\nu}= 1.1\|\eta(\rho^0)/\rho^0\|_{L^\infty}, \qquad
\bar{\lambda} = 1.1 \lambda/\underline{\rho}, \quad \text{ with } 1/\underline{\rho} = \|1/\rho^0\|_{L^\infty}.
\end{equation}
We note that our stability analysis also relies on additional restrictive conditions on the time step $\tau$ and density approximation $\rho^\tau$, see 
\eqref{eq:theorem1_hypo1}-\eqref{eq:theorem1_hypo2}. These conditions are similar to the one used in \cite{cappanera_vu_2024} to establish the stability of a momentum-based projection method. While being restrictive, as discussed in the following sections, these conditions do not need to be enforced in practice.

\subsection{Stability of the algorithm}
\label{sec:time_stability}

We establish the stability of the algorithm introduced in section \ref{sec:time_algo}. For the sake of simplicity, we assume that homogeneous Dirichlet boundary conditions are imposed on the momentum. As the density is bounded away from zero, it is equivalent to assuming that $\bu_{|\partial\Omega}=0$.

\begin{theorem} \label{thm:stab_bdf1}
Let $\gamma_1\in(0,\frac{1}{\sqrt{2}})$, 
$\gamma_2\in(0,\frac{1}{\sqrt{2} c_3})$ 
and $(\gamma_3,\gamma_4)\in(0,1)^2$
be constants. Let $\bar{\nu}$ and $\bar{\lambda}$ satisfy \eqref{eq:def_nu_rho_bar}. Assume that the approximation of the density $\rho^\tau$ is given by an algorithm of the form \eqref{eq:approximation_density} that satisfies the bound preserving principle \eqref{hyp:min_max_principle}. Assume that the time step $\tau$ satisfies for all $n\geq0$ the following conditions:
\begin{equation}\label{eq:theorem1_hypo1}
\tau^{1/2} \| \frac{\bar{\nu}2\sqrt{2}}{\sqrt{\rho^n\eta^{n+1}}}\GRAD \rho^{n+1} \|_{\bL^\infty} 
 \leq \gamma_1, \qquad
\tau^{1/2} \| \frac{\Bar{\lambda} \sqrt{2}}{\sqrt{\rho^n\eta^{n+1}}}\GRAD \rho^{n+1} \|_{\bL^\infty}
 \leq \gamma_2.
\end{equation}
Also assume that the sequence $\rho^\tau$ satisfies:
\begin{equation}\label{eq:theorem1_hypo2}
\|  \frac{\bar{\nu}\delta \rho^{n+1}}{\eta^{n+1}}\|_{L^\infty}
\leq \gamma_3, \qquad
\| \frac{\Bar{\lambda} \delta \rho^{n+1}}{\lambda} \|_{L^\infty} 
\leq \gamma_4.
\end{equation}
Then the sequences $(\bmom^\tau, p^\tau, \bu^\tau)$ defined by the scheme 
\eqref{eq:algo_step_mom}-\eqref{eq:algo_step_pre}-\eqref{eq:algo_step_vel}
satisfy the following stability inequality for all $n\geq0$:
\begin{multline}\label{eq:theo_1a}
\|\sqrt{\rho^{n+1}} \bu^{n+1}\|_{\bL^2}^2
+2\tau \bar{\nu} \|\sqrt{\rho^{n+1}} \varepsilon(\bu^{n+1}) \|_{\bL^2}^2
+\tau\Bar{\lambda} \|\sqrt{\rho^{n+1}} \DIV\bu^{n+1} \|_{L^2}^2
+\frac{\tau}{\lambda} \|p^{n+1}\|_{\bL^2}^2 
+\tau \|\sqrt{\eta^{n+1}} \varepsilon(\bu^{n+1}) \|_{\bL^2}^2
\\ \leq
\|\sqrt{\rho^{n}} \bu^{n}\|_{\bL^2}^2
+ 2\tau \bar{\nu} \|\sqrt{\rho^n} \varepsilon(\bu^{n}) \|_{\bL^2}^2
+\tau\Bar{\lambda} \|\sqrt{\rho^{n}} \DIV\bu^{n} \|_{L^2}^2
+\frac{\tau}{\lambda} \|p^{n}\|_{L^2}^2 
+ 2 \tau \| \bef^{n+1} \|_{\bL^2} \| \bu^{n+1} \|_{\bL^2}.
\end{multline}
As a result, there exists a constant $C$ independent of the time step such that:
\begin{equation}\label{eq:theo_1b}
\|\sqrt{\rho^{n+1}}  \bu^{n+1}\|_{\bL^2}^2
\leq
\|\sqrt{\rho^{0}}  \bu^{0}\|_{\bL^2}^2
+ \tau \bar{\nu} \| \sqrt{\rho^0} \varepsilon(\bu^0)  \|_{\bL^2}^2
+\tau \Bar{\lambda} \| \sqrt{\rho^0} \DIV\bu^0\|_{L^2}^2
+\frac{\tau}{\lambda} \|p^{0}\|_{L^2}^2 
+ C \tau \sum_{k=0}^{n} \| \bef^{k+1}\|_{\bL^2}^2.
\end{equation}
\end{theorem}

\begin{remark} We note that the hypothesis 
\ref{eq:theorem1_hypo1}-\ref{eq:theorem1_hypo2} are very restrictive.
Indeed, for very steep or unsmooth density profile, the gradient of density acts like $h^{-1}$ with $h$ the computational mesh size. 
As a consequence, the condition \ref{eq:theorem1_hypo1} can read $\tau \leq C h^2$ which is not desirable. 
We will show numerically in section \ref{sec:num_results} that the fully explicit algorithm introduced in section \ref{sec:final_algo} is stable under classic CFL condition ($\tau \propto h$) even for unsmooth density functions. 
Thus, the conditions \ref{eq:theorem1_hypo1}-\ref{eq:theorem1_hypo2} are heuristics and do not need to be enforced in practice.
\end{remark}

\begin{proof} 
We aim to test the momentum equation \eqref{eq:algo_step_mom} with $2\tau\bu^{n+1}$ and integrate over the domain $\Omega$. The following describes how to bound the resulting terms from the left hand side of the momentum equation
\eqref{eq:algo_step_mom}. 
We note that the term associated with the time derivative of the momentum satisfies the relation:
\[
2 \int_\Omega \delta \bmom^{n+1} \bu^{n+1} \diff\bx
= \delta \|\sqrt{\rho^{n+1}} \bu^{n+1} \|_{\bL^2}^2 
+ \|\sqrt{\rho^{n}} \delta \bu^{n+1} \|_{\bL^2}^2 
+ \int_\Omega \|\bu^{n+1}\|_{l^2}^2 \delta \rho^{n+1} \diff\bx.
 \]
Moreover, multiplying \eqref{eq:approximation_density} by $\tau \|\bu^{n+1}\|_{l^2}^2$ and integrating over $\Omega$ leads to the following identity:
\begin{equation}\notag
\int_\Omega \|\bu^{n+1}\|_{l^2}^2 \delta \rho^{n+1} \diff\bx
=
- \int_\Omega  2 \tau \bu^{n+1} \cdot \left(  
\DIV(\rho^n\bu^{n+1}{\otimes} \bu^n) 
- \frac{1}{2} \rho^n\bu^{n+1}\DIV \bu^{n}
\right) \diff \bx 
+ \tau\int_\Omega\|\bu^{n+1}\|_{l^2}^2 R_\rho^{n+1} \diff\bx.
\end{equation}
Noticing that 
$\varepsilon(\rho\bu):\varepsilon(\bv)\geq \rho \varepsilon(\bu):\varepsilon(\bv) - \|\bu\|_{l^2} \| \GRAD \rho\|_{l^\infty} \|\varepsilon(\bv)\|_{l^2}$ 
and that 
$\bmom^{n+1}-\bmom^{*,n}=\rho^{n+1}\delta\bu^{n+1}$,
the terms associated to the div-grad operators, i.e. the second and third term of the left handside of \eqref{eq:algo_step_mom}, satisfy:
$$
\begin{array}{ccl}
4 \tau \int_\Omega \bar{\nu} \varepsilon(\rho^{n+1} \delta\bu^{n+1})
:\varepsilon(\bu^{n+1}) \diff\bx
& \geq &
 2\tau \|\sqrt{\bar{\nu}\rho^{n+1}} \varepsilon(\bu^{n+1}) \|_{\bL^2}^2 
- 2\tau \|\sqrt{\bar{\nu}\rho^{n+1}} \varepsilon(\bu^{n}) \|_{\bL^2}^2 
+ 2\tau \|\sqrt{\bar{\nu}\rho^{n+1}} \varepsilon(\delta\bu^{n+1}) \|_{\bL^2}^2
\\ [1.2ex]
&  &
- 4 \tau \int_\Omega \bar{\nu} \|\delta\bu^{n+1}\|_{l^2} \| \GRAD\rho^{n+1}\|_{l^\infty}
\| \varepsilon(\bu^{n+1})\|_{l^2} \diff\bx
\\ [1.2ex]
& \geq &
2 \tau \|\sqrt{\bar{\nu}\rho^{n+1}} \varepsilon(\bu^{n+1}) \|_{\bL^2}^2 
- 2 \tau \|\sqrt{\bar{\nu}\rho^{n+1}} \varepsilon(\bu^{n}) \|_{\bL^2}^2 
+ 2 \tau \|\sqrt{\eta^{n+1}} \varepsilon(\delta\bu^{n+1}) \|_{\bL^2}^2
\\ [1.2ex]
& &
- 4 \tau \int_\Omega \bar{\nu} \|\delta\bu^{n+1}\|_{l^2} \| \GRAD\rho^{n+1}\|_{l^\infty}
\| \varepsilon(\bu^{n+1})\|_{l^2} \diff\bx,
\end{array}
$$
and
\begin{equation}\notag
4 \tau \int_\Omega \eta^{n+1} \varepsilon(\bu^{n}):\varepsilon(\bu^{n+1})\diff\bx
= 2\tau \|\sqrt{\eta^{n+1}} \varepsilon(\bu^{n+1}) \|_{\bL^2}^2 
+ 2 \tau \|\sqrt{\eta^{n+1}} \varepsilon(\bu^{n}) \|_{\bL^2}^2 
- 2\tau \|\sqrt{\eta^{n+1}} \varepsilon(\delta\bu^{n+1}) \|_{\bL^2}^2,
\end{equation}
where we used the definition of $\bar{\nu}:=\| \eta(\rho^0)/\rho^0\|_{\bL^\infty}$
and the maximum principle hypothesis \eqref{hyp:min_max_principle} to get the bound:
\[
\| \sqrt{\bar{\nu}\rho^{n+1}}\varepsilon(\delta\bu^{n+1}) \|_{\bL^2}^2  
\geq \|  \sqrt{\eta^{n+1}}\varepsilon(\delta\bu^{n+1}) \|_{\bL^2}^2.
\] 
Using the hypothesis $\eqref{eq:theorem1_hypo1}$ on the time step $\tau$, 
i.e. $\tau^{1/2} \| \frac{\bar{\nu}2\sqrt{2}}{\sqrt{\rho^n\eta^{n+1}}}\GRAD \rho^{n+1} \|_{\bL^\infty} 
 \leq \gamma_1$,
we infer that:
$$
\begin{array}{ccl}
4 \tau \int_\Omega \bar{\nu} \|\delta\bu^{n+1}\|_{l^2} \| \GRAD\rho^{n+1}\|_{l^\infty}
\| \varepsilon(\bu^{n+1})\|_{l^2} \diff\bx
& \leq &  
2\tau 
\| \frac{\bar{\nu}2 \sqrt{2}}{\sqrt{\rho^n\eta^{n+1}}}
\GRAD\rho^{n+1}\|_{\bL^\infty} 
\| \sqrt{\frac{\rho^n}{2}}  \delta \bu^{n+1}\|_{\bL^2}
 \| \sqrt{\eta^{n+1}}\varepsilon(\bu^{n+1})\|_{\bL^2}
\\ [1.3ex]
& \leq &
\| \sqrt{\frac{\rho^n}{2}} \delta \bu^{n+1} \|_{\bL^2}^2
+ \tau \gamma_1^2 \| \sqrt{\eta^{n+1}} \varepsilon(\bu^{n+1}) \|_{\bL^2}^2.
\end{array}
$$
Testing the terms associated to the grad-div operator with $2\tau\bu^{n+1}$, 
meaning the fourth and fifth terms of \eqref{eq:algo_step_mom}, reads:
$$
\begin{array}{ccl}
2 \tau \int_\Omega \Bar{\lambda} \DIV(\rho^{n+1}\delta\bu^{n+1})\DIV\bu^{n+1}   \diff\bx
& = &
\tau \| \sqrt{\Bar{\lambda} \rho^{n+1}}\DIV\bu^{n+1} \|_{\bL^2}^2 
- \tau \| \sqrt{\Bar{\lambda} \rho^{n+1}}\DIV\bu^{n}\|_{\bL^2}^2 
+ \tau \| \sqrt{\Bar{\lambda} \rho^{n+1}}\DIV \delta \bu^{n+1} \|_{\bL^2}^2 
\\ [1.2ex]
& &
+ 2 \tau \int_\Omega \Bar{\lambda}\delta \bu^{n+1} \cdot
\GRAD \rho^{n+1} \DIV\bu^{n+1}\diff\bx
\\ [1.2ex]
& \geq &
\tau \| \sqrt{\Bar{\lambda} \rho^{n+1}}\DIV\bu^{n+1} \|_{\bL^2}^2 
- \tau  \| \sqrt{\Bar{\lambda} \rho^{n+1}}\DIV\bu^{n} \|_{\bL^2}^2
+ \tau \lambda\| \DIV \delta\bu^{n+1}\|_{\bL^2}^2
\\ [1.2ex]
& &
- 2 \tau \int_\Omega \Bar{\lambda} \|\delta\bu^{n+1}\|_{l^2} 
\|\GRAD \rho^{n+1} \|_{l^\infty} 
|\DIV \bu^{n+1}| \diff \bx,
\end{array}
$$
and
\[
2 \tau \int_\Omega \lambda \DIV\bu^n\DIV\bu^{n+1} \diff\bx
=
\tau \lambda \left(
\| \DIV\bu^{n+1} \|_{\bL^2}^2 
+ \| \DIV\bu^{n}  \|_{\bL^2}^2 
- \|  \DIV\delta \bu^{n+1} \|_{\bL^2}^2 
\right),
\]
where the inequality:
\[ 
 \| \sqrt{\Bar{\lambda} \rho^{n+1}} \DIV \delta\bu^{n+1}\|_{\bL^2}^2
 \geq \lambda \| \DIV \delta\bu^{n+1}\|_{\bL^2}^2
\]
follows from the definition of $\underline{\rho}$, see \eqref{eq:def_nu_rho_bar},
and the maximum principle hypothesis \eqref{hyp:min_max_principle}. 
Notice that the restriction \eqref{eq:theorem1_hypo1}, i.e.
$\tau^{1/2} \| \frac{\Bar{\lambda}\sqrt{2}}{\sqrt{\rho^n\eta^{n+1}}}\GRAD\rho^{n+1}
 \|_{\bL^\infty} \leq \gamma_2$,
on the time step $\tau$ yields:
$$
\begin{array}{ccl}
2 \tau \int_\Omega \Bar{\lambda} \|\delta \bu^{n+1}\|_{l^2}
\|\GRAD \rho^{n+1}\|_{l^\infty}  | \DIV\bu^{n+1}| \diff\bx
& \leq &
2 \tau \| \tfrac{\Bar{\lambda} \sqrt{2}}{\sqrt{\rho^n\eta^{n+1}}} \GRAD \rho^{n+1} \|_{\bL^\infty} 
\|   \sqrt{\tfrac{\rho^n}{2}} \delta \bu^{n+1} \|_{\bL^2} 
\| \sqrt{\eta^{n+1}} \DIV\bu^{n+1} \|_{\bL^2}
\\ [1.3ex]
& \leq &
\| \sqrt{\tfrac{\rho^n}{2}} \delta \bu^{n+1}  \|_{\bL^2}^2   
+ \tau \gamma_2^2\|\sqrt{\eta^{n+1}} \DIV \bu^{n+1}  \|_{\bL^2}^2.
\end{array}
$$
Rewriting $\rho^{n+1}=\delta\rho^{n+1} +\rho^n$, the hypothesis \eqref{eq:theorem1_hypo2} yields the following inequalities:
\[ 
\| \sqrt{\bar{\nu}\rho^{n+1}} \varepsilon(\bu^n) \|_{\bL^2}^2
\leq 
\gamma_3 \|  \sqrt{\eta^{n+1}} \varepsilon(\bu^n) \|_{\bL^2}^2 
+ \| \sqrt{\bar{\nu}\rho^n} \varepsilon(\bu^n) \|_{\bL^2}^2,
\]
\[
\| \sqrt{\Bar{\lambda} \rho^{n+1}} \DIV\bu^n \|_{\bL^2}^2 
\leq 
\gamma_4 \lambda \|  \DIV\bu^n \|_{\bL^2}^2 
+ \| \sqrt{\Bar{\lambda} \rho^n} \DIV\bu^n \|_{\bL^2}^2.
\]
Eventually, using integration by parts and \eqref{eq:algo_step_pre}, 
the following relation holds:
\[
2 \tau \int_\Omega \GRAD p^n \cdot \bu^{n+1} \diff \bx 
=
\frac{\tau}{\lambda} \delta \| p^{n+1}\|_{\bL^2}^2 
- \tau \lambda \| \DIV\bu^{n+1}\|_{\bL^2}^2 .
\]
Combining all of the above identities and inequalities shows that testing the equation
\eqref{eq:algo_step_mom}  with $2\tau\bu^{n+1}$ reads:
$$
\begin{array}{ccl}
\delta \| \sqrt{\rho^{n+1} }\bu^{n+1}\|_{\bL^2}^2 
& + &
\frac{\tau}{\lambda} \delta \|p^{n+1}\|_{L^2}^2 
+ 2\tau \delta \| \sqrt{\bar{\nu}\rho^{n+1}} \varepsilon(\bu^{n+1}) \|_{\bL^2}^2 
+ \tau \Bar{\lambda} \delta \| \sqrt{\rho^{n+1}} \DIV\bu^{n+1} \|_{L^2}^2 
\\ [1.2ex]
& + &
\tau (2-\gamma_1^2-\gamma_2^2c_3^2) \| \sqrt{\eta^{n+1}} \varepsilon(\bu^{n+1}) \|_{\bL^2}^2 
+ 2\tau(1- \gamma_3) \| \sqrt{\eta^{n+1} \varepsilon(\bu^n)} \|_{\bL^2}^2 
+\tau \lambda (1-\gamma_4) \| \DIV \bu^n \|_{L^2}^2 
\\ [1.2ex]
& \leq &
2\tau \int_\Omega \bef^{n+1} \cdot \bu^{n+1} \diff \bx.
\end{array}
$$
where we used the relation $\|\sqrt{\eta^{n+1}} \DIV \bu^{n+1} \|_{\bL^2}
\leq c_3  \| \sqrt{\eta^{n+1}}\varepsilon(\bu^{n+1}) \|_{\bL^2} $ introduced in section
\ref{sec:notations_preliminaries}. We remind that $\gamma_1\in(0,\frac{1}{\sqrt{2}})$, 
$\gamma_2\in(0,\frac{1}{\sqrt{2}c_3})$
and $(\gamma_3,\gamma_4)\in(0,1)^2$, 
thus the terms involving these constants can all be dropped. 
The estimate \eqref{eq:theo_1a} follows readily using Cauchy-Schwarz inequality.
The energy bound \eqref{eq:theo_1b} is then the consequence of a standard telescopic argument combined with the Korn, Poincar\'e and Young inequalities.
\end{proof}

\subsection{Discussion on normalization of artificial compressibility parameter}
For single phase flow, artificial compressibility techniques are known to be efficient while setting $\lambda=1$. However for flows with variable density, setting $\lambda=1$ may not always yield a good control of the velocity divergence. For instance, this phenomena has been observed in \cite{lundgren2023high} which led the author to vary the value of $\lambda$ between $1$ and $5000$ depending on the setup (e.g. ratio of density, effective Reynolds number).
As we do not wish to tune the parameter $\lambda$ depending on the physical setting considered, we perform numerous numerical studies that led us to introduce the following rescaling of $\lambda$:
\begin{align}\label{eq:normalization_alpha}
\lambda =  \max(1,\bar{\nu} \underline{\rho} ) \lambda.
\end{align}
The idea behind this rescaling consists of renormalizing the parameter $\lambda$ such that the second stability condition in \eqref{eq:theorem1_hypo1}, that involves the quantity $\bar{\lambda} = \lambda / \underline{\rho}$, becomes analogous with the first stability condition of  \eqref{eq:theorem1_hypo1} that involves $\bar{\nu}$.
Thanks to this renormalization, all numerical simulations presented in the 
section~\ref{sec:num_results} are performed with $\lambda=1$ and allow us to recover expected stability and convergence properties.

\section{Time error analysis}
\label{sec:time_error_analysis}
In this section, we study the temporal convergence of the scheme \eqref{eq:algo_step_mom}-\eqref{eq:algo_step_vel} where we assume that the density is approximated with a first-order scheme that guarantees the min-max principle. Our main results are summarized in Theorem \ref{theorem:error_bound} and Corollary \ref{corollary_order_cvg}, where we show that the error in velocity is order half in time in $ l^\infty(0,T;\bL^2(\Omega))$ and in $ l^2(0,T;\bH^1(\Omega))$ norms. We will show in the next section, using numerical simulations, that the effective convergence rate is order one in time.

\subsection{Notations and hypothesis}
Let $W^{m,p}(\Omega)$ be the standard Sobolev space. We denote by $H^s(\Omega)$ the space $W^{s,2}(\Omega)$ and we set $H^0(\Omega):=L^2(\Omega)$.
We assume that the solution of \eqref{eq:mass_equation}-\eqref{eq:incompressibility_equation} has the following regularity:
\begin{equation}
\label{hyp:regularity solutions}
\rho\in L^\infty(0,T;L^\infty(\Omega))\cup W^{1,\infty}(0,T;L^\infty(\Omega)),
\; \bu\in \bL^\infty(0,T;\bW^{1,\infty}(\Omega))\cup \bW^{2,\infty}(0,T;\bL^\infty(\Omega)),\;
p\in  W^{1,\infty}(0,T;L^\infty(\Omega)).
\end{equation}
\label{sec:prelim_error_proof}
We define the following quantities related to the error in density, pressure and velocity:
\begin{equation}
\label{def:error_density}
e_\rho^n:=\rho(t_n)-\rho^n,
\end{equation}
\begin{equation}
\label{def:error_pressure}
e_p^n:=p(t_n)-p^n,
\end{equation}
\begin{equation}
\label{def:error_velocity}
\be_u^n:=\bu(t_n)-\bu^n.
\end{equation}
We also define the error for dynamical viscosity by:
\begin{equation}
\label{eq:5.1}
e_\eta^n:=\eta(t_n)-\eta^n.
\end{equation}
Using the Lipschitz hypothesis on $\eta$, 
see \eqref{hyp:Lipschitz_continuity}, 
there exists a positive constant $\alpha$ such that the following holds for all $0\leq n \leq N$:
\begin{equation}
\label{eq:approximation_error_density}
|e_\eta^n|\leq \alpha |e_\rho^n| .
\end{equation}
Eventually, as in the previous section, we focus our analysis on the Navier-Stokes part of the governing system \eqref{eq:Navier_Stokes}. That is, we study the convergence of the velocity-pressure couple. It leads us to assume that the scheme used to approximate the density, see \eqref{eq:algo_step_rho}, yields a residual $R_\rho$
and a density error $e_\rho$ that are both controlled in norm $l^2(0,T;L^2(\Omega))$ by the time step $\tau$ and the velocity error $\be_u$. Meaning we assume that there exists a positive constant $C$ independent of the time steps and the approximation such that the following holds for all $0\leq n \leq N$:
\begin{equation}
\label{hyp:residual_density}
\tau \sum_{k=0}^{n} \|R_\rho^k\|_{L^2}^2
\leq 
C\tau^2
+C \tau \sum_{k=0}^{n-1}\|\be_u^k\|_{\bL^2}^2,
\end{equation}
\begin{equation}
\label{hyp:error_density}
\tau \sum_{k=0}^{n}\|e_\rho^k\|_{L^2}^2
\leq 
C\tau^2
+C \tau \sum_{k=0}^{n-1}\|\be_u^k\|_{\bL^2}^2.
\end{equation}

\subsection{Preliminaries results}

We start our analysis by introducing an equivalent scheme, see Lemma \ref{lemma_equiv_scheme}, which uses the couple velocity-pressure as primary unknowns. Then, we establish error equations for the velocity and pressure, see Lemma \ref{lemma_vel_error_eq}-\ref{lemma_pre_error_eq}, that will be used in the following section to establish the convergence of the proposed method.

\begin{lemma}
\label{lemma_equiv_scheme}
Assume that the approximation of the density $\rho^\tau$ is given by an algorithm of the form \eqref{eq:approximation_density} that satisfies \eqref{hyp:min_max_principle}. Then, solving \eqref{eq:algo_step_rho}-\eqref{eq:algo_step_pre} with unknowns $(\bmom,p)$ is equivalent to solving the following scheme with unknowns $(\bu,p)$:
\begin{equation}
\label{eq:approximation_velocity}
\frac{\rho^n \delta \bu^{n+1}}{\tau}+\rho^n (\bu^n \cdot \nabla) \bu^{n+1}-2 \Bar{\nu}\nabla \cdot{( \varepsilon(\rho^{n+1}\delta \bu^{n+1})}-2\nabla \cdot (\eta^{n+1} \varepsilon(\bu^n))-\Bar{\lambda}\nabla (\nabla \cdot (\rho^{n+1} \delta \bu^{n+1}))
\end{equation}
\begin{equation}\notag
-\lambda \nabla (\nabla \cdot \bu^n)+\nabla p^n+\frac{1}{2}\rho^n\bu^{n+1}\nabla \cdot \bu^n+\frac{1}{2} R_\rho^{n+1}\bu^{n+1}=\bf{f}^{n+1},
\end{equation}
where the pressure is still determined using \eqref{eq:algo_step_pre}.
\end{lemma}

\begin{proof}
Observe that \eqref{eq:algo_step_rho} implies:
\begin{equation}
\label{eq:3.4}
\frac{\rho^{n+1}\bu^{n+1}}{\tau} =
\frac{\rho^n \bu^{n+1}}{\tau} -(\bu^n\cdot\nabla\rho^n)\bu^{n+1}+R_\rho^{n+1}\bu^{n+1}.
\end{equation}
By expanding the term $\DIV(\rho^n\bu^{n+1}{\otimes} \bu^n) ,$ and using $\bmom^n=\rho^n \bu^n$, the equation \eqref{eq:algo_step_mom} can be rewritten as:
\begin{equation}
\label{eq:3.5}
 \frac{\rho^{n+1}\bu^{n+1}-\rho^n \bu^n}{\tau}+\rho^n (\bu^n \cdot \nabla) \bu^{n+1} +(\bu^n \cdot \nabla \rho^n)\bu^{n+1}+\rho^n\bu^{n+1}(\nabla\cdot \bu^n)-2 \Bar{\nu}\nabla \cdot{( \varepsilon(\rho^{n+1}\delta \bu^{n+1})}
\end{equation}
\begin{equation}\notag
 -2\nabla \cdot (\eta^{n+1} \varepsilon(\bu^n))-\Bar{\lambda}\nabla (\nabla \cdot (\rho^{n+1}\delta \bu^{n+1}))-\lambda \nabla(\nabla \cdot \bu^n)+\nabla p^n-\frac{1}{2}\rho^n(\nabla\cdot \bu^n)\bu^{n+1}-\frac{1}{2} R_\rho^{n+1}\bu^{n+1}=\bf{f}^{n+1}.
\end{equation}
We conclude by substituting \eqref{eq:3.4} into \eqref{eq:3.5}.
\end{proof}

We now establish the error equation for the velocity in the following lemma.
\begin{lemma}
\label{lemma_vel_error_eq}
Assuming that the solutions of \eqref{eq:mass_equation}-\eqref{eq:incompressibility_equation} satisfy the regularity hypothesis \eqref{hyp:regularity solutions}, and under the hypothesis of Lemma \ref{lemma_equiv_scheme}, the following holds for all integers $0\leq n \leq N$:
\begin{multline}
\label{eq:lemma_error_velocity}
\|\sqrt{\rho^{n+1}} \be_u^{n+1}\|_{\bL^2}^2
-\|\sqrt{\rho^n}\be_u^n\|_{\bL^2}^2
+\|\sqrt{\rho^n}\delta \be_u^{n+1}\|_{\bL^2}^2
+4\tau\|\sqrt{\eta^{n+1}}\varepsilon(\be_u^{n+1})\|_{\bL^2}^2
\\
= -2\tau \inner{R_u^{n+1}}{\be_u^{n+1}}
-2\tau \inner{ e_\rho^{n+1}(u(t_{n+1})\cdot \nabla)\bu(t_{n+1})}{\be_u^{n+1}}
+2\tau \inner{((\rho^n\bu(t_n)-\rho^{n+1}\bu(t_{n+1}))\cdot \nabla)\bu(t_{n+1})}{\be_u^{n+1}}
\\
-2\tau \inner{\rho^n (\be_u^n \cdot \nabla) \bu(t_{n+1})} {\be_u^{n+1}}
-4\tau \inner{e_\eta^{n+1}\varepsilon(\bu(t_{n+1}))}{\varepsilon(\be_u^{n+1})}
+4\tau \inner{\Bar {\nu}(\varepsilon(\rho^{n+1} \delta \bu^{n+1}))}{\varepsilon(\be_u^{n+1})}
\\
-4\tau \inner{ (\eta^{n+1}\varepsilon(\delta \bu^{n+1})}{\varepsilon(\be_u^{n+1})}
-2\tau\Bar{\lambda}\inner{\nabla(\nabla \cdot (\rho^{n+1} \delta \bu^{n+1}))}{\be_u^{n+1}}
-2\tau\inner{\nabla p(t_{n+1})}{\be_u^{n+1}}
\\
-2\tau\inner{\lambda\nabla(\nabla \cdot \bu^n)}{\be_u^{n+1}}
+2\tau\inner{\nabla p^n}{\be_u^{n+1}}+\tau\inner{\rho^n(\nabla\cdot \bu^n)u(t_{n+1})}{\be_u^{n+1}}
+\tau\inner{R_\rho^{n+1}\bu(t_{n+1})}{\be_u^{n+1}},
\end{multline}
where
\begin{equation}
\label{eq:lemma_error_velocity_def_Ru}
\begin{array}{crl}
R_\bu^{n+1}
& := &  
\rho(t_{n+1}) \partial_t \bu(t_{n+1})-\rho^n \frac{\bu(t_{n+1})-\bu(t_n)}{\tau}
\\[1.2ex]
& = &
\partial_t \bu(t_{n+1})(\rho(t_{n+1})-\rho(t_n))-\rho^n \left(\frac{\bu(t_{n+1})-\bu(t_n)}{\tau}-\partial_t \bu(t_{n+1})\right)+ e_\rho^n \partial_t \bu(t_{n+1}).
\end{array}
\end{equation}
\end{lemma}

\begin{proof} 
Using equation \eqref{eq:mass_equation}, one can show that equation \eqref{eq:momentum_equation} can be rewritten as follows:
\begin{equation}
\label{eq:5.5}
\rho \partial_{t} \bu
+\rho(\bu \cdot \nabla)\bu
-2 \nabla \cdot (\eta \varepsilon(\bu))
+\nabla p 
=\bf{f}.
\end{equation}
By taking the difference between equation \eqref{eq:5.5} at time $t_{n+1}$ and equation \eqref{eq:approximation_velocity}, we obtain:
\begin{multline}\notag
\rho^n \frac{\delta \be_u^{n+1}}{\tau}
+ R_\bu^{n+1}
+\rho(t_{n+1})(\bu(t_{n+1}) \cdot \nabla) \bu(t_{n+1})-2\nabla\cdot(\eta(t_{n+1}) \varepsilon(\bu(t_{n+1}))
+\nabla p(t_{n+1})
\\
+\rho^n(\be_u^n \cdot \nabla) \bu^{n+1} 
-\rho^n(\bu(t_n) \cdot \nabla)\bu^{n+1}
+2 \Bar{\nu} \nabla \cdot(\varepsilon(\rho^{n+1} \delta \bu^{n+1}))
+2\nabla \cdot (\eta^{n+1} \varepsilon(\bu^n))
+\Bar{\lambda}\nabla (\nabla \cdot (\rho^{n+1}\delta \bu^{n+1}))
\\
+\lambda \nabla(\nabla \cdot \bu^n)-\nabla p^n-\frac{1}{2}\rho^n(\nabla \cdot \bu^n)\bu^{n+1}-\frac{1}{2} R_\rho^{n+1}\bu^{n+1}=0.
\end{multline}

Multiplying the above equation by $2\tau \be_u^{n+1}$ and integrating over the domain $\Omega$, we get:
\begin{multline}\label{eq:error_eq_u_step_1}
\|\sqrt{\rho^n}\be_u^{n+1}\|_{\bL^2}^2
-\|\sqrt{\rho^n}\be_u^n\|_{\bL^2}^2
+\|\sqrt{\rho^n}\delta \be_u^{n+1}\|_{\bL^2}^2
+ 2\tau \inner{R_\bu^{n+1}}{\be_u^{n+1}}
+ 2\tau \inner{\rho(t_{n+1})(\bu(t_{n+1})\cdot \nabla)\bu(t_{n+1})}{\be_u^{n+1}}
\\
+2\tau \inner{\rho^n(\be_u^n \cdot \nabla) \bu^{n+1}} {\be_u^{n+1}} 
-2\tau \inner{\rho^n (\bu(t_n) \cdot \nabla) \bu^{n+1}} {\be_u^{n+1}} 
-4\tau \inner{\nabla \cdot (\eta(t_{n+1}) \varepsilon(\bu(t_{n+1}))}{\be_u^{n+1}}
\\
+4\tau \inner{\Bar {\nu} \nabla \cdot(\varepsilon(\rho^{n+1}\delta \bu^{n+1}))}{\be_u^{n+1}}
+4\tau \inner{\nabla \cdot (\eta^{n+1}\varepsilon(\bu^n)}{\be_u^{n+1}}
+2\tau\Bar{\lambda}\inner{\nabla(\nabla \cdot (\rho^{n+1} \delta \bu^{n+1}))}{\be_u^{n+1}}
\\
+2\tau \inner{\nabla p(t_{n+1})}{\be_u^{n+1}} 
+2\tau\lambda\inner{\nabla(\nabla \cdot \bu^n)}{\be_u^{n+1}}
-2\tau\inner{\nabla p^n}{\be_u^{n+1}}-\tau\inner{\rho^n(\nabla \cdot \bu^n)\bu^{n+1}}{\be_u^{n+1}}
-\tau\inner{R_\rho^{n+1}\bu^{n+1}}{\be_u^{n+1}}=0.
\end{multline}
Multiplying \eqref{eq:algo_step_rho} by $\tau |\be_u^{n+1}|^2$ and integrating over $\Omega$ reads:
\begin{equation}\notag
\inner{\frac{\delta \rho^{n+1}}{\tau}}{\tau |\be_u^{n+1}|^2}
+ \inner{\bu^n\cdot\nabla\rho^n}{\tau |\be_u^{n+1}|^2}
=\tau\inner{R_\rho^{n+1}}{ |\be_u^{n+1}|^2}.
\end{equation}
It follows that:
\begin{equation}\notag
\|\sqrt{\rho^n}\be_u^{n+1}\|_{\bL^2}^2
=
\|\sqrt{\rho^{n+1}} \be_u^{n+1}\|_{\bL^2}^2
+\tau\inner{(\bu^n \cdot \nabla\rho^n) \be_u^{n+1}}{\be_u^{n+1}}
-\tau\inner{R_\rho^{n+1}}{ |\be_u^{n+1}|^2}.
\end{equation}
Substituting the above relation into \eqref{eq:error_eq_u_step_1} and integrating by parts the terms multiplied by the viscosity $\eta$ or $\bar{\nu}$, we obtain:
\begin{multline} \label{eq:error_eq_u_step_2}
\|\sqrt{\rho^{n+1}} \be_u^{n+1}\|_{\bL^2}^2-\|\sqrt{\rho^n}\be_u^n\|_{\bL^2}^2
+\|\sqrt{\rho^n}\delta \be_u^{n+1}\|_{\bL^2}^2
\\
=
- 2\tau \inner{R_\bu^{n+1}}{\be_u^{n+1}}
- 2\tau \inner{\rho(t_{n+1})(\bu(t_{n+1})\cdot \nabla)\bu(t_{n+1})}{\be_u^{n+1}}
-\tau\inner{(\bu^n \cdot \nabla\rho^n) \be_u^{n+1}}{\be_u^{n+1}}
\\
-2\tau \inner{\rho^n (\be_u^n \cdot\nabla) \bu^{n+1}} {e_u^{n+1}} 
+2\tau \inner{\rho^n (\bu(t_n) \cdot\nabla) \bu^{n+1}} {\be_u^{n+1}} 
-4\tau \inner{\eta(t_{n+1}) \varepsilon(\bu(t_{n+1})}{\varepsilon(\be_u^{n+1})}
\\
+4\tau \inner{\Bar {\nu} \varepsilon(\rho^{n+1}\delta \bu^{n+1})}{\varepsilon(\be_u^{n+1})}
+4\tau \inner{\eta^{n+1}\varepsilon(\bu^n)}{\varepsilon(\be_u^{n+1})}
-2\tau\Bar{\lambda} \inner{\nabla(\nabla \cdot (\rho^{n+1} \delta \bu^{n+1}))}{\be_u^{n+1}}
\\
-2\tau \inner{\nabla p(t_{n+1})}{\be_u^{n+1}} 
-2\tau\lambda\inner{\nabla(\nabla\cdot\bu^n)}{\be_u^{n+1}}
+2\tau\inner{\nabla p^n}{\be_u^{n+1}}
+\tau \inner{\rho^n(\nabla \cdot \bu^n) \bu^{n+1}}{\be_u^{n+1}}
+\tau\inner{R_\rho^{n+1}\bu(t_{n+1})}{\be_u^{n+1}}
.
\end{multline}
We conclude by observing that the following identities hold:
\begin{equation}\notag
-4\tau \inner{\eta^{n+1}\varepsilon(\bu^n)}{\varepsilon(\be_u^{n+1})}
=-4\tau \inner{\eta^{n+1}\varepsilon(\bu(t_{n+1}))}{\varepsilon(\be_u^{n+1})}
+4\tau \inner{\eta^{n+1}\varepsilon(\be_u^{n+1})}{\varepsilon(\be_u^{n+1})}
+4\tau \inner{\eta^{n+1}\varepsilon(\delta \bu^{n+1})}{\varepsilon(\be_u^{n+1})},
\end{equation}
\begin{equation}\notag
4\tau \inner{\eta(t_{n+1})\varepsilon(\bu(t_{n+1}))}{\varepsilon(\be_u^{n+1})}
-4\tau \inner{\eta^{n+1}\varepsilon(\bu(t_{n+1}))}{\varepsilon(\be_u^{n+1})}
=4\tau \inner{e_\eta^{n+1}\varepsilon(\bu(t_{n+1}))}{\varepsilon(\be_u^{n+1})},
\end{equation}
\begin{equation}\notag
-2\tau \inner{\rho(t_{n+1})(\bu(t_{n+1})\cdot \nabla)\bu(t_{n+1})}{\be_u^{n+1}}=
-2\tau \inner{e_\rho^{n+1}(\bu(t_{n+1})\cdot \nabla)\bu(t_{n+1})}{\be_u^{n+1}}
-2\tau \inner{\rho^{n+1}(\bu(t_{n+1})\cdot \nabla)\bu(t_{n+1})}{\be_u^{n+1}},
\end{equation}
and
\begin{multline}\notag
-\tau\inner{(\bu^n \cdot \nabla\rho^n) \be_u^{n+1}}{\be_u^{n+1}}
-2\tau \inner{\rho^n (\be_u^n \cdot \nabla) \bu^{n+1}} {\be_u^{n+1}} 
+2\tau \inner{\rho^n (\bu(t_n) \cdot \nabla) \bu^{n+1}} {\be_u^{n+1}} 
\\
=
-2\tau \inner{\rho^n (\be_u^n \cdot \nabla) \bu(t_{n+1})} {\be_u^{n+1}}
+\tau\inner{\rho^n(\nabla \cdot \bu^n)
\be_u^{n+1}}{\be_u^{n+1}}
+2\tau \inner{\rho^n (\bu(t_n) \cdot \nabla) \bu(t_{n+1})} {\be_u^{n+1}}.
\end{multline}
Indeed, combining the above relations with \eqref{eq:error_eq_u_step_2} yields \eqref{eq:lemma_error_velocity}.
\end{proof}

We conclude this section with the derivation of an equation for the pressure error that will later be combined with \eqref{eq:lemma_error_velocity} when studying the convergence of the algorithm \ref{eq:algo_step_mom}-\eqref{eq:algo_step_vel}.
\begin{lemma}
\label{lemma_pre_error_eq} 
Under the assumptions of lemma \ref{lemma_vel_error_eq}, the following holds for all integers $0\leq n \leq N$:
\begin{multline}
\label{eq:lemma_error_pressure}
-2\tau\inner{\nabla p(t_{n+1})}{\be_u^{n+1}}
-2\tau\inner{\lambda\nabla(\nabla \cdot \bu^n)}{\be_u^{n+1}}
+2\tau\inner{\nabla p^n}{\be_u^{n+1}}
\\
=
-\frac{\tau}{\lambda}\|e_p^{n+1}\|_{L^2}^2
+\frac{\tau}{\lambda}\|e_p^n\|_{L^2}^2
-\frac{\tau}{\lambda}\|e_p^{n+1}-e_p^n\|_{L^2}^2
+\frac{2\tau}{\lambda}\inner{e_p^{n+1}}{p(t_{n+1})-p(t_n)}
-2\tau\inner{\lambda\nabla(\nabla \cdot \delta \be_u^{n+1})}{\be_u^{n+1}}.
\end{multline}
\end{lemma}
\begin{proof}
Taking the gradient of \eqref{eq:algo_step_pre} reads:
\begin{equation}\notag
\nabla (p^n-p^{n+1})=\lambda\nabla(\nabla\cdot \bu^{n+1}).
\end{equation}
Multiplying the above relation with $2\tau \be_u^{n+1}$, integrating over the domain $\Omega$ and adding the result to the last term of the fourth line and first two terms of the fifth line from \eqref{eq:lemma_error_velocity}
gives:
\begin{equation}\notag
-2\tau\inner{\nabla p(t_{n+1})}{\be_u^{n+1}}
-2\tau\inner{\lambda\nabla(\nabla \cdot \bu^n)}{\be_u^{n+1}}
+2\tau\inner{\nabla p^n}{\be_u^{n+1}}
=
-2\tau\inner{\nabla e_p^{n+1}}{\be_u^{n+1}}
-2\tau\inner{\lambda\nabla(\nabla \cdot \delta \be_u^{n+1})}{\be_u^{n+1}}.
\end{equation}
Observing that:
$$
\begin{array}{rcl}
-2\tau\inner{\nabla e_p^{n+1}}{\be_u^{n+1}}
&=&
2\tau\inner{ e_p^{n+1}}{\nabla\cdot \be_u^{n+1}}
\\  [1.2ex]
&=&
-2\tau\inner{ e_p^{n+1}}{\nabla\cdot \bu^{n+1}}
\\ [1.2ex]
&=&
\frac{2\tau}{\lambda}\inner{e_p^{n+1}}{p^{n+1}-p^n}
\\ [1.2ex]
&=&
-\frac{2\tau}{\lambda}\inner{e_p^{n+1}}{e_p^{n+1}-e_p^n}
+\frac{2\tau}{\lambda}\inner{e_p^{n+1}}{p(t_{n+1})-p(t_n)},
\end{array}
$$
we conclude by applying the polarizing identity.
\end{proof}

\subsection{Main results}
We now show the main result of this section, namely that the velocity approximation $\bu^\tau$ converges with order half in $l^\infty(0,T;\bL^2(\Omega))$ and in $l^2(0,T;\bH^1(\Omega))$ norms. We conclude this section with a remark that hints how to improve the velocity error bound to first order and to get a pressure error bound of order half in $l^\infty(0,T;H^1(\Omega))$.
In the remainder of this section, we denote by $C$ and $\epsilon_i$ positive constants that are independent of the time step $\tau$ and the approximations $(\rho^\tau, \bu^\tau, p^\tau)$.

\begin{theorem}
\label{theorem:error_bound}
Let $\gamma$ be a positive constant such that $\gamma^2\leq {\min\left(\frac{1}{4},\frac{1}{64 \Bar{\nu}}\right)}$. 
Assume that the solutions of \eqref{eq:mass_equation}-\eqref{eq:incompressibility_equation} satisfy the regularity hypothesis \eqref{hyp:regularity solutions}.
Moreover, assume that the approximation of the density $\rho^\tau$ is given by an algorithm of the form \eqref{eq:approximation_density} that satisfies for all $0\leq n \leq N$, the maximum principle \eqref{hyp:min_max_principle} and the following conditions:
\begin{equation}\label{eq:lemma3_hypo1}
(\tau\bar{\nu})^{1/2} \| \frac{\GRAD \rho^{n+1}}{\sqrt{\rho^n\eta^{n+1}}} \|_{\bL^\infty} 
 \leq \gamma, \qquad
\frac{(\tau\bar{\lambda})^{1/2} }{\underline{\rho}}\| \frac{\GRAD \rho^{n+1}}{\sqrt{\rho^n\eta^{n+1}}} \|_{\bL^\infty} 
 \leq \gamma,
\end{equation}
\begin{equation}\label{eq:lemma3_hypo2}
\|  \frac{\bar{\nu}\delta \rho^{n+1}}{\eta^{n+1}}\|_{L^\infty}
\leq 1, \qquad
\| \frac{\Bar{\lambda} \delta \rho^{n+1}}{\lambda} \|_{L^\infty} 
\leq 1.
\end{equation}
Then, the sequences $(\bu^\tau, p^\tau)$ generated by the algorithm \eqref{eq:algo_step_mom}-\eqref{eq:algo_step_pre} satisfy the following error estimates for all $0\leq n \leq N$:
\begin{equation}
\label{eq:main_results}
(1-C\tau) A^{n+1} 
+ \tau B^{n+1} 
\leq 
(1+C\tau) A^n 
+C\tau^2
+C\tau \left(\|e_\rho^{n+1}\|_{L^2}^2
+\|e_\rho^n\|_{L^2}^2
+\|R_\rho^{n+1}\|_{L^2}^2\right)
,
\end{equation}
where 
\begin{equation}\notag
A^n:=\|\sqrt{\rho^n} \be_u^n\|_{\bL^2}^2
+\frac{\tau}{\lambda}\|e_p^n\|_{L^2}^2
+2\tau\Bar{\nu}\|\sqrt{\rho^n} \varepsilon (\be_u^n)\|_{\bL^2}^2
+ \tau\Bar{\lambda}\|\sqrt{\rho^n} \DIV\be_u^n\|_{\bL^2}^2
,
\end{equation}
and
\begin{equation}\notag
B^n:=
\|\sqrt{\eta^n} \varepsilon (\be_u^n)\|_{\bL^2}^2,
\end{equation}
\end{theorem}

\begin{proof}
Using lemmas \ref{lemma_vel_error_eq}-\ref{lemma_pre_error_eq}, we inject \eqref{eq:lemma_error_pressure} into \eqref{eq:lemma_error_velocity} to obtain:
\begin{multline}
\label{eq:main_results_proof1_1}
\|\sqrt{\rho^{n+1}} \be_u^{n+1}\|_{\bL^2}^2
-\|\sqrt{\rho^n}\be_u^n\|_{\bL^2}^2
+\|\sqrt{\rho^n}\delta \be_u^{n+1}\|_{\bL^2}^2+
\frac{\tau}{\lambda}\|{e_p^{n+1}}\|_{L^2}^2-\frac{\tau}{\lambda}\|e_p^n\|_{L^2}^2+\frac{\tau}{\lambda}\|e_p^{n+1}-{e_p^n}\|_{L^2}^2
+4\tau\|\sqrt{\eta^{n+1}}\varepsilon(\be_u^{n+1})\|_{\bL^2}^2
\\
= -2\tau \inner{R_\bu^{n+1}}{\be_u^{n+1}}
-2\tau \inner{ e_\rho^{n+1}(u(t_{n+1})\cdot \nabla)\bu(t_{n+1})}{\be_u^{n+1}}
+2\tau \inner{((\rho^n\bu(t_n)-\rho^{n+1}\bu(t_{n+1}))\cdot \nabla)\bu(t_{n+1})}{\be_u^{n+1}}
\\
-2\tau \inner{\rho^n(\be_u^n \cdot \nabla) \bu(t_{n+1})} {\be_u^{n+1}}
-4\tau \inner{(e_\eta^{n+1}\varepsilon(\bu(t_{n+1}))}{\varepsilon(\be_u^{n+1})}
+4\tau \inner{\Bar {\nu}(\varepsilon(\rho^{n+1} \delta \bu^{n+1}))}{\varepsilon(\be_u^{n+1})}
\\
-4\tau \inner{ (\eta^{n+1}\varepsilon(\delta \bu^{n+1})}{\varepsilon(\be_u^{n+1})}
-2\tau\Bar{\lambda} \inner{\nabla(\nabla \cdot (\rho^{n+1} \delta \bu^{n+1}))}{\be_u^{n+1}}
-2\tau\inner{\lambda\nabla((\nabla \cdot \delta \be_u^{n+1})}{\be_u^{n+1}}
\\
+\frac{2\tau}{\lambda}\inner{e_p^{n+1}}{p(t_{n+1})-p(t_n)}+\tau\inner{\rho^n(\nabla\cdot \bu^n)\bu(t_{n+1})}{e_u^{n+1}}
+\tau\inner{R_\rho^{n+1}\bu(t_{n+1})}{\be_u^{n+1}}=\sum_{i=1}^{11}A_i.
\end{multline}
We now establish bounds for each term $A_i$, whose definition is written below for completeness. We note that, in the following, we may use Korn inequality \eqref{eq:prelim_Korn}, Poincar\'e inequality \eqref{eq:prelim_Poincare}, Young inequalities and polarization identity without mentioning them explicitly.
\begin{equation}\notag
\begin{array}{cclccl}
A_1
& := &  
-2\tau \inner{R_\bu^{n+1}}{\be_u^{n+1}},
& 
A_7
& := &
-2\tau\Bar{\lambda}\inner{\nabla(\nabla \cdot (\rho^{n+1} \delta \bu^{n+1}))}{\be_u^{n+1}},
\\ [1.2ex]
A_2
& := & 
-2\tau \inner{ e_\rho^{n+1}(\bu(t_{n+1})\cdot \nabla)\bu(t_{n+1})}{\be_u^{n+1}},
& 
A_8
& := &
-2\tau\inner{\lambda\nabla(\nabla \cdot \delta \be_u^{n+1})}{\be_u^{n+1}},
\\ [1.2ex]
A_3
& := & 
2\tau \inner{((\rho^n\bu(t_n)-\rho^{n+1}\bu(t_{n+1}))\cdot \nabla)\bu(t_{n+1})}{\be_u^{n+1}},
& 
A_9
& := &
\frac{2\tau}{\lambda}\inner{e_p^{n+1}}{p(t_{n+1})-p(t_n)},
\\ [1.2ex]
A_4
& := & 
-2\tau \inner{\rho^n (\be_u^n \cdot \nabla) \bu(t_{n+1})} {\be_u^{n+1}},
& 
A_{10}
& := &
\tau\inner{\rho^n(\nabla\cdot \bu^n)\bu(t_{n+1})}{\be_u^{n+1}},
\\ [1.2ex]
A_5
& := & 
-4\tau \inner{(e_\eta^{n+1}\varepsilon(\bu(t_{n+1})}{\varepsilon(\be_u^{n+1})},
& 
A_{11}
& := &
\tau\inner{R_\rho^{n+1} \bu(t_{n+1})}{\be_u^{n+1}}.
\\ [1.2ex]
A_6
& := & 
4\tau \inner{\Bar {\nu}(\varepsilon(\rho^{n+1} \delta \bu^{n+1}))}{\varepsilon(\be_u^{n+1})}
-4\tau \inner{ (\eta^{n+1}\varepsilon(\delta \bu^{n+1})}{\varepsilon(\be_u^{n+1})},
& 
& &
\\ [1.2ex]
\end{array}
\end{equation}
Observe that using \eqref{eq:lemma_error_velocity_def_Ru} and the regularity hypothesis \eqref{hyp:regularity solutions} gives:
$$\|R_\bu^{n+1}\|_{L^2}^2 \leq C(\tau^2 +\|e_\rho^n\|_{L^2}^2).$$
Therefore, applying Cauchy-Schwarz we bound the terms $A_1$ and $A_2$ as follows:
$$A_1
\leq 
C \tau^3 + C \tau (\|e_\rho^n\|_{L^2}^2 + \|e_\bu^n\|_{\bL^2}^2),$$
$$A_2 
\leq
C \tau ( \|e_\rho^n\|_{L^2}^2 + \|e_\bu^n\|_{\bL^2}^2 ) .
$$

We write the term $A_3$ as follows:
\begin{equation}\notag
A_3=2\tau \inner{((\rho^n(\bu(t_n)-\bu(t_{n+1}))\cdot \nabla)\bu(t_{n+1})}{\be_u^{n+1}}-2\tau \inner{((\rho^{n+1}-\rho^n)\bu(t_{n+1})\cdot \nabla)\bu(t_{n+1})}{\be_u^{n+1}}:=I_1+I_2.
\end{equation}
Using the regularity hypothesis \eqref{hyp:regularity solutions}, we get
\begin{equation}\notag
I_1\leq C\tau^3+C\tau\|\be_u^{n+1}\|_{\bL^2}^2,
\end{equation}
and \begin{equation}\notag
\begin{array}{ccl}
I_2 
& = &
-2\tau \inner{(e_\rho^n \bu(t_{n+1})\cdot\nabla)\bu(t_{n+1}))}{\be_u^{n+1}}
+ 2\tau \inner{(e_\rho^{n+1} \bu(t_{n+1})\cdot\nabla)\bu(t_{n+1}))}{\be_u^{n+1}}
\\ [1.2ex]
& &
-2\tau \inner{(\rho(t_{n+1})-\rho(t_n))\bu(t_{n+1})\cdot\nabla)\bu(t_{n+1}))}{\be_u^{n+1}} 
\\ [1.2ex]
& \leq &
C\tau \left(
\|e_\rho^n\|_{L^2}^2
+ \|e_\rho^{n+1}\|_{L^2}^2 
+\|\be_u^{n+1}\|_{\bL^2}^2
\right)
+C\tau^3.
\end{array}
\end{equation}

Similarly, using the regularity hypothesis \eqref{hyp:regularity solutions} the Lipschitz condition on the dynamical viscosity, see \eqref{eq:approximation_error_density}, we get
$$
A_4
\leq
C \tau (\|\be_u^{n}\|_{\bL^2}^2 + \|\be_u^{n+1}\|_{\bL^2}^2)
$$
$$
A_5
\leq
C\tau\|e_\rho^{n+1}\|_{L^2}^2
+
\tau\epsilon_1\|\varepsilon(\be_u^{n+1})\|_{\bL^2}^2,
$$
where we use Young's inequality to get the existence of a positive constant $\epsilon_1$ that is independent of the time step and approximation.
Hence, so far, we have
\begin{equation}\notag
\sum_{i=1}^{5}A_i
\leq
C\tau \left(\|\be_u^{n+1}\|_{\bL^2}^2+\|e_\rho^{n+1}\|_{L^2}^2
+\|\be_u^n\|_{\bL^2}^2
+\|e_\rho^n\|_{L^2}^2
\right)
+\tau\epsilon_1\|\varepsilon(\be_u^{n+1})\|_{\bL^2}^2
+C\tau^3.
\end{equation}

We now focus our attention on the term $A_6$. First, we rewrite $A_6$ as follows:
\begin{equation}\notag
\begin{array}{crl}
A_6
& = &
4\tau \Bar {\nu}\inner{\rho^{n+1}\varepsilon( \delta \bu^{n+1}))}{\varepsilon(\be_u^{n+1})}
\\ [1.2ex]
& &
+2\tau \Bar {\nu}\inner{\delta \bu^{n+1}\otimes \nabla \rho^{n+1}+\nabla \rho^{n+1}\otimes \delta \bu^{n+1}}{\varepsilon(\be_u^{n+1})}
\\ [1.2ex]
& &
-4\tau \inner{ (\eta^{n+1}\varepsilon(\delta \bu^{n+1})}{\varepsilon(\be_u^{n+1})}
\\ [1.2ex]
& = & 
B_1 + B_2 + B_3.
\end{array}
\end{equation}
Notice that
\begin{equation}\notag
\begin{array}{ccl}
B_1
& = & 
-4\tau \Bar {\nu}\inner{\rho^{n+1}\varepsilon( \delta \be_u^{n+1}))}{\varepsilon(\be_u^{n+1})}
+4\tau \Bar {\nu}\inner{\rho^{n+1}\varepsilon(\bu(t_{n+1})-\bu(t_n))}{\varepsilon(\be_u^{n+1})}
\\
& = &
-2\tau\Bar{\nu}\|\sqrt{\rho^{n+1}}\varepsilon(\be_u^{n+1})\|_{\bL^2}^2
+2\tau\Bar{\nu}\|\sqrt{\rho^{n+1}}\varepsilon(\be_u^n)\|_{\bL^2}^2
-2\tau\Bar{\nu}\|\sqrt{\rho^{n+1}}\varepsilon(\delta \be_u^{n+1})\|_{\bL^2}^2
\\
& &
+4\tau \Bar {\nu}\inner{\rho^{n+1}\varepsilon(\bu(t_{n+1})-\bu(t_n))}{\varepsilon(\be_u^{n+1})}.
\end{array}
\end{equation}
Using \eqref{eq:lemma3_hypo2}, we get
\begin{equation}\notag
\Bar{\nu}\|\sqrt{\rho^{n+1}}\varepsilon(\be_u^n)\|_{\bL^2}^2
\leq \Bar{\nu}\|\sqrt{\rho^n}
\varepsilon(\be_u^n)\|_{\bL^2}^2+\|\sqrt{\eta^{n+1}}\varepsilon(\be_u^n)\|_{\bL^2}^2.
\end{equation}
Therefore, using \eqref{hyp:regularity solutions} and Young inequality, we bound $B_2$ as follows:
\begin{equation}\notag
B_1\leq -2\tau\Bar{\nu}\|\sqrt{\rho^{n+1}}\varepsilon(\be_u^{n+1})\|_{\bL^2}^2
+2\tau\Bar{\nu}\|\sqrt{\rho^n}\varepsilon(\be_u^n)\|_{\bL^2}^2
+2\tau\|\sqrt{\eta^{n+1}}\varepsilon(\be_u^n)\|_{\bL^2}^2
\end{equation}
\begin{equation}\notag
-2\tau\Bar{\nu}\|\sqrt{\rho^{n+1}}\varepsilon(\delta \be_u^{n+1})\|_{\bL^2}^2
+\tau\epsilon_2\|\varepsilon(\be_u^{n+1})\|_{\bL^2}^2
+C\tau^3,
\end{equation}
where $\epsilon_2$ is a positive constant.
Similarly, we bound the term $B_2$ as follows:
\begin{equation}
\label{eq:error_bound_limit_B2}
\begin{array}{ccl}
B_2
& \leq &
4\tau \Bar{\nu} \|\sqrt{\rho^n} \delta \be_u^{n+1}\|_{\bL^2}
\|\frac{1}{\sqrt{\rho^n\eta^{n+1}}} \nabla\rho^{n+1}\|_{\bL^\infty}^2
 \|\sqrt{\eta^{n+1}}\varepsilon(\be_u^{n+1})\|_{\bL^2}
\\ [1.2ex]
& &
+4\tau \Bar{\nu} \|\sqrt{\rho^n} (\bu(t_{n+1})-\bu(t_n))\|_{\bL^2}
\|\frac{1}{\sqrt{\rho^n\eta^{n+1}}} \nabla\rho^{n+1}\|_{\bL^\infty}^2
 \|\sqrt{\eta^{n+1}}\varepsilon(\be_u^{n+1})\|_{\bL^2}
\\ [1.2ex]
& \leq &
\frac{2\gamma^2}{\epsilon_3} \|\sqrt{\rho^n} \delta \be_u^{n+1}\|_{\bL^2}^2
+2(\epsilon_3+\epsilon_4)\Bar{\nu}\tau
 \|\sqrt{\eta^{n+1}}\varepsilon(\be_u^{n+1})\|_{\bL^2}^2
 +C\tau^2,
\end{array}
\end{equation}
where we use hypothesis \eqref{eq:lemma3_hypo1} to get the last inequality.
Using regularity hypothesis \eqref{hyp:regularity solutions}, maximum principle \eqref{hyp:min_max_principle_eta}, and observing that
\begin{equation}\notag
B_3=4\tau \inner{ (\eta^{n+1}\varepsilon(\delta \be_u^{n+1})}{\varepsilon(\be_u^{n+1})}-4\tau \inner{ (\eta^{n+1}\varepsilon(\bu(t_{n+1})-\bu(t_n))}{\varepsilon(\be_u^{n+1})},
\end{equation}
we obtain:
\begin{equation}\notag
B_3 \leq 2\tau\|\sqrt{\eta^{n+1}}\varepsilon(\be_u^{n+1})\|_{\bL^2}^2
-2\tau\|\sqrt{\eta^{n+1}}\varepsilon(\be_u^n))\|_{\bL^2}^2
+2\tau\|\sqrt{\eta^{n+1}}\varepsilon(\delta \be_u^{n+1})\|_{\bL^2}^2
+\tau\epsilon_5\|\varepsilon(\be_u^{n+1})\|_{\bL^2}^2
+C\tau^3.
\end{equation}
Hence, we get:
\begin{equation}\notag
\begin{array}{ccl}
A_6 = B_1 + B_2 + B_3
& \leq &  
\frac{2\gamma^2}{\epsilon_3}\|\sqrt{\rho^n}\delta \be_u^{n+1}\|_{\bL^2}^2
+\tau(2+2\Bar{\nu}\epsilon_3+2\Bar{\nu}\epsilon_4)\|\sqrt{\eta^{n+1}}\varepsilon(\be_u^{n+1}))\|_{\bL^2}^2
+C\tau^2
\\ [1.2ex]
& & 
-2\tau\Bar{\nu}\|\sqrt{\rho^{n+1}}\varepsilon(\be_u^{n+1})\|_{\bL^2}^2
+2\tau\Bar{\nu}\|\sqrt{\rho^n}\varepsilon(\be_u^n)\|_{\bL^2}^2
+\tau(\epsilon_2+\epsilon_5)\|\varepsilon(\be_u^{n+1})\|_{\bL^2}^2.
\end{array}
\end{equation}

Then, we write the term $A_7$ as follows:
$$
\begin{array}{ccl}
A_7
& = &
2\tau\Bar{\lambda}\inner{(\nabla \cdot (\rho^{n+1} \delta \bu^{n+1})}{\nabla \cdot \be_u^{n+1}}
\\
& = &  
-2\tau\Bar{\lambda}\inner{(\nabla \cdot (\rho^{n+1} \delta \be_u^{n+1})}{\nabla \cdot \be_u^{n+1}}+2\tau \Bar{\lambda}\inner{(\nabla \cdot (\rho^{n+1} (\bu(t_{n+1})-\bu(t_n)))}{\nabla \cdot \be_u^{n+1}}
\\
& = &  
D_1+D_2.
\end{array}
$$
Noticing that
\begin{equation}\notag
D_1
=
-2\tau\Bar{\lambda}\inner{\delta \be_u^{n+1}\cdot \nabla\rho^{n+1})}{\nabla \cdot \be_u^{n+1}}
-2\tau\Bar{\lambda}\inner{\rho^{n+1}(\nabla \cdot \delta \be_u^{n+1})}{\nabla \cdot \be_u^{n+1}},
\end{equation}
using \eqref{eq:lemma3_hypo1}, Young inequality, and polarization identity, we obtain
$$
\begin{array}{ccl}
D_1
& \leq &  
\frac{\gamma^2}{\epsilon_6}\|\sqrt{\rho^n}\delta \be_u^{n+1}\|_{\bL^2}^2
+ \tau \epsilon_6 \|\sqrt{\eta^{n+1}}\nabla \cdot\be_u^{n+1}\|_{L^2}^2
-\tau\Bar{\lambda}\| \sqrt{\rho^{n+1}}\nabla\cdot \be_u^{n+1}\|_{L^2}^2
\\
& & 
+\tau\Bar{\lambda}\| \sqrt{\rho^{n+1}}\nabla\cdot \be_u^n\|_{L^2}^2
-\tau\Bar{\lambda}\| \sqrt{\rho^{n+1}}\nabla\cdot \delta \be_u^{n+1}\|_{L^2}^2.
\\ 
\end{array}
$$
\begin{equation}\notag
D_1\leq 
\frac{\gamma^2}{\epsilon_6}\|\sqrt{\rho^n}\delta \be_u^{n+1}\|_{\bL^2}^2+ \tau \epsilon_6 \|\sqrt{\eta^{n+1}}\nabla \cdot\be_u^{n+1}\|_{L^2}^2
-\tau\Bar{\lambda}\| \sqrt{\rho^{n+1}}\nabla\cdot \be_u^{n+1}\|_{L^2}^2
\end{equation}
\begin{equation}\notag
+\tau\Bar{\lambda}\| \sqrt{\rho^{n+1}}\nabla\cdot \be_u^n\|_{L^2}^2-\tau\Bar{\lambda}\| \sqrt{\rho^{n+1}}\nabla\cdot \delta \be_u^{n+1}\|_{L^2}^2.
\end{equation}
Using \eqref{eq:lemma3_hypo2}, we get
$$
\begin{array}{ccl}
\tau\Bar{\lambda}\| \sqrt{\rho^{n+1}}\nabla\cdot \be_u^n\|_{L^2}^2
& \leq &  
\tau\Bar{\lambda}\|\delta \rho^{n+1}\|_{L^\infty}\|\nabla\cdot \be_u^n\|_{L^2}^2+\tau\Bar{\lambda}\| \sqrt{\rho^n}\nabla\cdot \be_u^n\|_{L^2}^2
\\
& \leq &
\tau\lambda\|\nabla\cdot \be_u^n\|_{L^2}^2+\tau\Bar{\lambda}\| \sqrt{\rho^n}\nabla\cdot \be_u^n\|_{L^2}^2.
\end{array}
$$
Hence, similarly than the term $B_2$, we bound $A_7$ as follows:
\begin{equation}
\label{eq:error_bound_limit_A7}
\begin{array}{ccl}
A_7 = D_ 1 + D_2
& \leq &  
\frac{ \gamma^2}{\epsilon_6}\|\sqrt{\rho^n}\delta \be_u^{n+1}\|_{\bL^2}^2
+\tau \epsilon_6 \|\sqrt{\eta^{n+1}}\nabla \cdot\be_u^{n+1}\|_{L^2}^2
-\tau\Bar{\lambda}\| \sqrt{\rho^{n+1}}\nabla\cdot \be_u^{n+1}\|_{L^2}^2
+\tau\lambda\|\nabla\cdot \be_u^n\|_{L^2}^2
\\
& &  
+\tau\Bar{\lambda}\| \sqrt{\rho^n}\nabla\cdot \be_u^n\|_{L^2}^2
-\tau\Bar{\lambda}\| \sqrt{\rho^{n+1}}\nabla\cdot \delta \be_u^{n+1}\|_{L^2}^2
+\tau\epsilon_7\| \nabla\cdot \be_u^{n+1}\|_{L^2}^2
+C\tau^2,
\end{array}
\end{equation}
where we use the hypothesis \eqref{hyp:regularity solutions}-\eqref{eq:lemma3_hypo1} to bound the term 
$|\nabla \cdot (\rho^{n+1} (\bu(t_{n+1})-\bu(t_n)))|$ 
in $D_2$ with $\tau^2$.
The remaining terms $A_i$, for $i=8,9,10,11$, are bounded using the polarization identity, the Young inequality, and hypothesis \eqref{hyp:min_max_principle}-\eqref{hyp:regularity solutions}. 
It reads:
\begin{equation}
\label{eq:error_bound_limit_A9}
\begin{array}{ccl}
A_8
&=&
\tau \lambda\|\nabla\cdot \be_u^{n+1}\|_{L^2}^2
-\tau \lambda\|\nabla\cdot \be_u^n\|_{L^2}^2
+\tau \lambda\|\nabla\cdot \delta \be_u^{n+1}\|_{L^2}^2,
\\
A_9
& \leq &  
C\frac{\tau^2}{\lambda}\|e_p^{n+1}\|_{L^2}^2+C\tau^2,
\\ [1.2ex]
A_{10}
& \leq &  
\tau\epsilon_{8}\|\nabla\cdot\delta \be_u^{n+1}\|_{L^2}^2
+ \tau\epsilon_{9} \|\varepsilon(\be_u^{n+1})\|_{\bL^2}^2
+C\tau\|\be_u^{n+1}\|_{\bL^2}^2,
\\ [1.2ex]
A_{11}
& \leq &  
C \tau (\|R_\rho^{n+1}\|_{L^2}^2 +\|\be_u^{n+1}\|_{\bL^2}^2).
\end{array}
\end{equation}


We now rewrite the term $\|e_p^{n+1}-{e_p^n}\|_{L^2}^2$ from the left handside of \eqref{eq:main_results_proof1_1} as follows:
$$
\begin{array}{ccl}
\frac{\tau}{\lambda}\|e_p^{n+1}-e_p^n\|_{L^2}^2
& = &  
\frac{\tau}{\lambda}\inner{e_p^{n+1}-e_p^n}{e_p^{n+1}-e_p^n}
\\
& = &  
\frac{\tau}{\lambda}\inner{p(t_{n+1})-p(t_n)}{e_p^{n+1}-e_p^n}-\frac{\tau}{\lambda}\inner{p^{n+1}-p^n}{e_p^{n+1}-e_p^n}
\\
& = &
E_1+E_2.
\end{array}
$$
Note that
\begin{equation}
\label{eq:error_bound_limit_E1}
| E_1 |
\leq 
C\tau^2
+C\tau(\frac{\tau}{\lambda}\|e_p^{n+1}\|_{L^2}^2
+\frac{\tau}{\lambda}\|e_p^n\|_{L^2}^2),
\end{equation}
and
$$
\begin{array}{ccl}
E_2
& = &  
\tau\inner{\nabla\cdot \bu^{n+1}}{e_p^{n+1}-e_p^n}
\\
& = &  
-\tau\inner{\nabla\cdot \be_u^{n+1}}{p(t_{n+1})-p(t_n) - (p^{n+1}-p^n)  }
\\
& = &
-\tau\inner{\nabla\cdot \be_u^{n+1}}{p(t_{n+1})-p(t_n)}
- \tau \lambda \inner{\nabla\cdot \be_u^{n+1}}{\nabla \cdot \bu^{n+1}}
\\
& \geq &  
- ( \tau \epsilon_{10} \|\nabla\cdot\be_u^{n+1}\|_{L^2}^2+C\tau^3)
+\tau\lambda\|\nabla\cdot \be_u^{n+1}\|_{L^2}^2
\end{array}
$$
where the bound on $E_2$ results from the incompressibility of $\bu(t^{n+1})$, the regularity hypothesis \eqref{hyp:regularity solutions}, and the artificial compressibility step of the algorithm, see \eqref{eq:algo_step_pre}. Moreover, using \eqref{eq:prelim_div} we get:
$$ 
\|\nabla\cdot\be_u^{n+1}\|_{L^2}^2
\leq
\frac{c_3}{\eta_{\min}}
\|\sqrt{\eta^{n+1}}\varepsilon(\be_u^{n+1})\|_{\bL^2}^2
.
$$

Combining the above bounds and injecting them into \eqref{eq:main_results_proof1_1} reads:
\begin{multline}
\label{eq:main_results_proof1_2}
(1 - C \tau) 
    \|\sqrt{\rho^{n+1}} \be_u^{n+1}\|_{\bL^2}^2
+ \frac{\tau}{\lambda} (1 - C \tau)
    \|{e_p^{n+1}}\|_{L^2}^2
+ 2\tau\Bar{\nu}\|\sqrt{\rho^{n+1}}\varepsilon(\be_u^{n+1})\|_{\bL^2}^2
+\tau\Bar{\lambda}\| \sqrt{\rho^{n+1}}\nabla\cdot \be_u^{n+1}\|_{L^2}^2
\\
+ \tau \left( 2 
    - \frac{\epsilon_1+\epsilon_2+\epsilon_5+\epsilon_9}{\eta_{\min}} 
    - 2\Bar{\nu} ( \epsilon_3+\epsilon_4) - \epsilon_6
    - \frac{c_3(\epsilon_7+\epsilon_{10})}{\eta_{\min}}
    \right)
    \|\sqrt{\eta^{n+1}}\varepsilon(\be_u^{n+1})\|_{\bL^2}^2
\\
+ (1-\frac{2\gamma^2}{\epsilon_3} - \frac{ \gamma^2}{\epsilon_6})
    \|\sqrt{\rho^n}\delta \be_u^{n+1}\|_{\bL^2}^2
+\tau ( \Bar{\lambda} - \frac{\lambda+ \epsilon_8}{\rho_{\min}})
    \| \sqrt{\rho^{n+1}}\nabla\cdot \delta \be_u^{n+1}\|_{L^2}^2
\\
\leq 
(1+C\tau) 
    \|\sqrt{\rho^n}\be_u^n\|_{\bL^2}^2
+ \frac{\tau}{\lambda}\|e_p^n\|_{L^2}^2
+ 2\tau\Bar{\nu}\|\sqrt{\rho^{n}}\varepsilon(\be_u^{n})\|_{\bL^2}^2
+\tau\Bar{\lambda}\| \sqrt{\rho^n}\nabla\cdot \be_u^n\|_{L^2}^2
\\
+ C \tau \left( \| e_\rho^{n+1}\|_{L^2}^2
+\| e_\rho^{n}\|_{L^2}^2
+\|R_\rho^{n+1}\|_{L^2}^2 \right)
+ C \tau^2
,
\end{multline}
where we use the minimum-maximum principle \eqref{hyp:min_max_principle}-\eqref{hyp:min_max_principle_eta} to get inequalities of the form:  
$$
\|\varepsilon(\be_u^{n+1})\|_{\bL^2}^2
\leq \frac{1}{\eta_{\min}}
\|\sqrt{\eta^{n+1}}\varepsilon(\be_u^{n+1})\|_{\bL^2}^2
, \quad
 \| \nabla\cdot \delta \be_u^{n+1}\|_{L^2}^2
\leq
\frac{1}{\rho_{\min}}\| \sqrt{\rho^{n+1}}\nabla\cdot \delta \be_u^{n+1}\|_{L^2}^2
, \quad
 \| \nabla\cdot \be_u^{n+1}\|_{L^2}^2
\leq
\frac{1}{\eta_{\min}}\| \sqrt{\eta^{n+1}}\nabla\cdot \be_u^{n+1}\|_{L^2}^2
.
$$
We remind that $\epsilon_i$ are positive constant independent of the approximation and time step, thus they can be chosen arbritrarly. 
Setting
$$\epsilon_1 = \epsilon_2 = \epsilon_5 = \epsilon_9 = \frac{\eta_{\min}}{16}
, \quad
\epsilon_3 = \epsilon_4 = \frac{1}{16\Bar{\nu}}
, \quad
\epsilon_6 = \frac{1}{4}
, \quad
\epsilon_7=\epsilon_{10} = \frac{\eta_{\min}}{8c_3}
, \quad
\epsilon_8 = \Bar{\lambda} - \frac{\lambda}{\rho_{\min}}
$$
and using the bounds \eqref{eq:lemma3_hypo1} on $\gamma$ leads to \eqref{eq:main_results} which concludes the proof.
\end{proof}

\begin{corollary}
\label{corollary_order_cvg}
Under the assumption of Theorem \ref{theorem:error_bound}, assuming that the time step $\tau$ is chosen such that $C \tau<1/2$, we have:
\begin{equation}
\label{eq:main_results_order_cvg}
\| \be_u\|_{l^\infty(0,T;\bL^2(\Omega))}^2
+ \| \be_u \|_{l^2(0,T;\bH^1(\Omega))}^2 
\leq \tilde{C} \tau
,
\end{equation}
where $\tilde{C}$ is a positive constant independent of the time step and the approximations $\rho^\tau, \bu^\tau, p^\tau$.
\end{corollary}

\begin{proof}
First, we sum the inequality \eqref{eq:main_results} at iteration $k$ from $k=0$ till $n-1$. It reads:
$$
(1 - C \tau) \sum_{k=0}^n A^{k} 
+ \tau \sum_{k=0}^n B^k
\leq
\sum_{k=0}^{n-1} (1 + C \tau) A^k
+ C \tau
+ C \tau \sum_{k=0}^{n-1}
    \left(
    \| e_\rho^{k+1}\|^2 + \| e_\rho^k\|^2 + \|R_\rho^{k+1}\|^2
    \right)
,
$$
where we use that $A^0=B^0=0$ and that $\sum_{k=0}^{n-1} \tau^2 \leq \sum_{k=0}^N \tau^2 = T \tau$. Using the hypothesis \ref{hyp:residual_density}-\ref{hyp:error_density}, we can rewrite the above inequality as follows:
$$
(1 - C \tau) \sum_{k=1}^n A^{k} 
+ \tau \sum_{k=1}^n B^k
\leq
\sum_{k=0}^{n-1} (1 + C \tau) A^k
+ C \tau.
$$
Using that $0< C \tau <1/2$, we get $1 \leq \frac{1}{1-C\tau}\leq 2$.
Thus, dividing the above inequality by $1-C\tau$ yields:
$$
\sum_{k=0}^n A^{k} 
+ \tau \sum_{k=0}^n B^k
\leq
\sum_{k=0}^{n-1} (1+ 4 C \tau) A^k
+ 2 C \tau
.
$$
The error estimate \eqref{eq:main_results_order_cvg} is then a consequence from discrete Gr\"{o}nwall lemma.
\end{proof}

\begin{remark}
We note that the above results yield an order half-convergence for the velocity and do not provide order of convergence for the pressure. These limitations are the consequence of the presence of a term in $\tau^2$ in the inequalities \eqref{eq:error_bound_limit_B2}-\eqref{eq:error_bound_limit_A7}-\eqref{eq:error_bound_limit_A9}-\eqref{eq:error_bound_limit_E1}.
The author expect that the results of \ref{corollary_order_cvg} can be improved by making more restrictive assumptions on the density approximation and the exact pressure. 
First, following \cite{cappanera_vu_2024}, one can assume that $\GRAD \rho^\tau$ is a bounded sequence in $l^\infty(0,T;L^\infty(\Omega))$ to improve the bounds \eqref{eq:error_bound_limit_B2}-\eqref{eq:error_bound_limit_A7} to get a first order approximation. 
Then, following \cite{shen1995error} one can introduce condition on the pressure (see hypothesis A2 and proof of lemma 3.3 therein) that are expected to improve the term $C\tau^2$ to $C\tau^3$ when bounding pressure terms related to $A_9$ and $E_1$, see\eqref{eq:error_bound_limit_A9}-\eqref{eq:error_bound_limit_E1}. 
The authors did not follow these directions because it would have made the analysis more lengthy while assuming more restrictive conditions on the density approximation that are not relevant for immiscible multi-fluids. Indeed, for such problems the density's gradient scales like the inverse of the mesh size  and can not be bounded uniformly. We note that the numerical simulations presented in the following section show that the proposed method does converge with order one in time with an hyperbolic CFL condition (i.e. time step proportional to the mesh size) for all the unknowns even for problems where the density function is unsmooth.
\end{remark}

\section{Numerical Illustrations} 
\label{sec:num_results}
In this section, we study the convergence properties of the proposed artificial compressibility technique for incompressible flows with variable density and viscosity. After introducing an equivalent formulation that uses a level set technique to track the dynamics of the density, we test our semi-implicit scheme with manufactured solutions on two dimensional domains using the software FreeFEM++. Eventually, we investigate the properties of an explicit version of our scheme that is suitable for spectral and pseudo-spectral methods. The explicit version is validated on three dimensional problems using the software SFEMaNS. We note that the algorithms are tested on similar setups that have been used to validate a momentum based projection method that we introduced in \cite{cappanera_2018,cappanera_vu_2024}.

\subsection{Level set method}
\label{sec:num_level_set}
We rewrite the mass conservation equation \eqref{eq:mass_equation} using a level set representation technique. The method consists of introducing a function $\phi$ defined in $[0,T]\times\Omega$ that takes value in $[0,1]$ such that the density is a function of the level set $\phi$ that satisfies
\begin{equation}
\label{eq:level_set_equation}
\partial_t \phi + \bu \cdot \GRAD \phi = 0.
\end{equation}
In practice, we will use a linear dependency between the density and the level set. In the context of immiscible multiphase flow with variable density, denoting by $\rho_{\min}$ and $\rho_{\max}$ the minimum and maximum value of the density, this relation can be define as follows:
\begin{equation}
\label{eq:def_level_set}
\phi = \frac{\rho - \rho_{\min}}{\rho_{\max} - \rho_{\min}}.
\end{equation}
Due to the linear relation between $\rho$ and $\phi$, it is equivalent to solve the mass conservation equation problem \eqref{eq:mass_equation} for the density or for the above level set function.  We note that the relation between $\phi$ and $\rho$ does not need to be linear. For instance, a popular type of level set method consists of introducing a smooth level set with constant gradient equal to one and to use a distance function to $\phi^{-1}(0.5)$ to reconstruct the density. We refer to \cite{osher2001level} for a review on level set techniques.

The time discretization of the level set, that replaces \eqref{eq:algo_step_rho}, then reads as follows. Given $\phi^n$ and $\bu^n$, finds $\phi^{n+1}$ such that:
\begin{equation}
\label{eq:time_level_set}
\frac{\phi^{n+1} - \phi^n}{\tau} 
+ \bu^n \cdot \GRAD \phi^{n+1}
- \DIV(\nu_{h} \GRAD \phi^{n+1})
= 0,
\end{equation}
where $\nu_h$ is an artificial viscosity that is added to stabilize the algorithm. 
In the rest of the paper, we define $\nu_h$ as a first-order viscosity, meaning that it is made proportional to the local mesh size $h$. Although such stabilization method is known to be strongly diffusive, it has the advantage to be robust and first order, meaning it is consistent with the order of our scheme. We note that this formulation, that is equivalent to set $R_\rho^{n+1}= \DIV(\nu_{h} \GRAD \phi^{n+1})$ in \eqref{eq:algo_step_rho}, does not guarantee that the level set, and so the density, conserves its maximum-minimum principle. Such structure preserving techniques could be constructed using flux-transport corrected technology \cite{boris1973flux,chiu2011conservative,guermond2017conservative,kuzmin2012flux,zalesak1979fully} and less diffusive artificial viscosity \cite{nazarov2013residual,stiernstrom2021residual}. 
While such methods are out of the scope of this paper, we note that for problems with immiscible multifluids, such as the one studied in sections \ref{sec:num_SFEMaNS_nonsmooth_axisym} and \ref{sec:num_MPR_NST}, we apply a compression technique inspired from \cite{Kreiss}. The idea of this technique, originally introduced in \cite{HARTENIII} as a level set correction technique, consists of adding a term in \eqref{eq:time_level_set} that balances the diffusion effect of artificial viscosity. As a result, the level set keeps its original sharp profile near the fluids interface. For the framework considered here, where the level set has value in $[0,1]$ and the fluids interface is represented by $\phi=1/2$, the algorithm reads:
\begin{equation}
\label{eq:time_level_set_compression}
\frac{\phi^{n+1} - \phi^n}{\tau} 
+ \bu^n \cdot \GRAD \phi^{n+1}
- \DIV(\nu_{h} \GRAD \phi^{n+1})
+ \DIV \left(c_{comp} \nu_h h^{-1} \phi^n (1-\phi^n) 
\frac{\GRAD \phi^n}{||\GRAD \phi^n||_{l^2}} \right)
= 0,
\end{equation}
where $h$ is the local mesh size and $c_{comp}$ is a tunable parameter that is either set to 1 for immiscible fluids (sharp level set) and 0 for other problems (smooth level set for manufactured solutions). We refer to \cite{COUPEZ_NUMIFORM2007,COUPEZ2010,guermond2017conservative,cappanera_2018} for more information on this compression technique.

\subsection{Weak formulation for semi-implicit algorithm}
\label{sec:final_algo}

We consider two types of spatial discretization. For two-dimensional problems, used to validate the semi-implicit scheme \eqref{eq:algo_step_rho}-\eqref{eq:algo_step_pre}, we use a continuous Galerkin finite element method. For three-dimensional problems, used to validate an explicit version of the scheme, we use a pseudo-spectral representation where Fourier decomposition is used in the azimuthal direction $\theta$ and finite elements are used in a meridian plane $(r,z)$. In the second case, unknowns can be written as a Fourier series in $\cos(m \theta)$ and $sin(m\theta)$ where the $m$-th Fourier coefficient lives in a two-dimensional finite element space. Thus, to simplify notation, we only describe the spatial discretization for problem in $\Omega \subset \mathbb{R}^2$ using finite element. 

We introduce a mesh sequence $(\mathcal{E}_h)_{h}$ of the domain $\Omega$ that is conforming, shape regular and consists of simplex elements. The mesh size is defined as $h=\max_{K\in\mathcal{E}_h} h_K$, where $h_K$ represents the diameter of a cell $K \in \mathcal{E}_h$.
The density, momentum, and pressure are approximated in the following finite element spaces
\begin{align} \label{eq:space_FEM}
X_h & =\{\psi\in \calC^0(\Omega;\Real)\mid 
\psi_{|K} \in \polP_{2},\ \forall K\in \mathcal{E}_h\}, \\
\bX_h & =\{\bv\in \calC^0(\Omega;\Real^2)\mid 
\bv_{|K} \in \bpolP_2,\ \forall K\in \mathcal{E}_h\}, \\
M_h & =\{q\in \calC^0(\Omega;\Real)\mid 
q_{|K} \in \polP_{1},\ \forall K\in \mathcal{E}_h\},
\end{align}  
where $\polP_k$ represents the vector space of polynomials functions with a total degree
of at most $k$. The dynamical viscosity $\eta$, and velocity $\bu$, are approximated in $X_h$, and $\bX_h$, respectively. Notice that the couple momentum-pressure is approximated with Taylor-Hood finite element $\bpolP_2-\polP_{1}$. 

The fully discrete algorithm for the semi-implicit scheme \eqref{eq:algo_step_rho}-\eqref{eq:algo_step_pre} reads as follows. After initialization, given $(\phi^n, \rho^n, \eta^n, \bmom^n, \bu^n, p^n)$ for an integer $n\geq 0$, we first find $\phi^{n+1}\in X_h$ such that
\begin{multline}
\label{eq:final_scheme_level_set}
\int_\Omega  
\frac{\phi^{n+1} - \phi^n}{\tau} \psi
\diff\bx
+ \int_\Omega ( \bu^n \cdot \GRAD \phi^{n+1} )\psi \diff\bx
+ \int_\Omega \nu_h \GRAD \phi^{n+1} \cdot \GRAD \psi \diff\bx
- \int_\Omega \left(c_{comp} \nu_h h^{-1} \phi^n (1-\phi^n) 
\frac{\GRAD \phi^n}{||\GRAD \phi^n||_{l^2}} \right) \GRAD \psi \diff\bx
=
0,
\end{multline}
holds for all  $\psi \in X_h$.

The second step consists of computing the density $\rho^{n+1}$ and dynamical viscosity $\eta^{n+1}$ using that $\rho$ is a linear function of $\phi$, see \eqref{eq:def_level_set}, and that $\eta$ is a given function of the density. In practice, we also define $\eta$ as a linear function of the level set $\phi$.

For the third step, we find $\bmom^{n+1} \in \bX_h$ such that the following holds for all $\bv \in \bX_h$:
\begin{multline}
\label{eq:final_scheme_momentum}
 \int_\Omega \left(
 \frac{1}{\dt}(\bmom^{n+1}-\bmom^n)\SCAL\bv 
+ 2 \bar{\nu} \varepsilon(\bmom^{n+1}-\bmom^{*,n}){:}\varepsilon(\bv) 
+ \lambda \DIV (\bmom^{n+1}-\bmom^{*,n}) \DIV \bv
+ ( \bu^n \cdot \nabla)\bmom^{n+1} \cdot \bv
\right) \diff \bx
\\ = 
\int_\Omega \left(
-2\eta^{n+1}\varepsilon(\bu^n){:}\varepsilon(\bv) 
-\GRAD p^n \SCAL \bv  
- \lambda \DIV \bu^n \DIV \bv
+ \bef^{n+1} \SCAL\bv
\right)\diff\bx,
\end{multline}
where $\bmom^{*,n} = \rho^{n+1} \bu^n$.

Eventually, the velocity and pressure are computed as follows:
\begin{equation}
\label{eq:final_scheme_u_and_p}
\bu^{n+1} = \frac{1}{\rho^{n+1}}, \qquad \qquad p^{n+1} = p^n - \lambda \DIV \bu^{n+1}.
\end{equation}

\subsection{Tests for semi-implicit scheme with finite element method}
This section investigate numerically the stability and convergence properties of the above semi-implicit scheme. We consider two setups with manufactured solutions introduced by one of the authors in \cite{cappanera_vu_2024} to validate a momentum based projection method. Our results are consistent with the theoretical properties established in sections \ref{sec:stab_algo} and \ref{sec:time_error_analysis}. In fact, numerical results suggest that the scheme yields better properties as the scheme remain stable under classic CFL condition, i.e. $\tau \sim h$, and we observe a first order convergence. We also note that these results can be compared with Tables 5 and 9 of \cite{cappanera_vu_2024}. They show that the proposed artificial compressibility technique, while presenting the advantage of a straightforward pressure update that does not require solving a Poisson problem, behaves similarly than the projection method described in \cite{cappanera_vu_2024}.
All the simulations presented in this section are done using the finite element software FreeFEM++. We refer to \cite{hecht2012new} for more information on this software.

\subsubsection{Test with $\eta(\rho)$ linear}
First, we consider a setup inspired from \cite{li2021} and introduced in \cite{cappanera_vu_2024} to validate a pressure-correction projection method for incompressible flows with variable density and viscosity. We consider the Navier-Stokes equations \eqref{eq:Navier_Stokes} on the tine interval $[0,1]$ and set the domain to the unit disk $\Omega=\{(x,y)\in \mathbb{R}^2 ; \; |x^2+y^2| < 1\}$. The velocity, level set and pressure are defined as follows: 
\[\bu(x,y,t) = \left(\frac{3}{4} + \frac{1}{4}\sin(t) \right)\begin{bmatrix}
    -\sin^2(x)\sin(y)\cos(y)\\
    \sin(x)\cos(x)\sin^2(y)
\end{bmatrix},\]
\[p(x,y,t) =\sin(x)\sin(y)\sin(t), \]
\[\phi(x,y,t) = \frac{1}{2} + \frac{1}{2} \sqrt{x^2+y^2}\cos(\theta - \sin(0.5 t)).\]
The dynamical viscosity is assumed to be a linear function of the density, which is itself a linear function of the level set. Thus, we define them as follows:
\[\rho(x,y,t) =  F(\phi) = \rho_1 + (\rho_2-\rho_1) \phi(x,y,t),\]
\[\eta(x,y,t) = \eta(\rho) =  \eta_1 + \frac{\eta_2-\eta_1}{\rho_2-\rho_1}(\rho(x,y,t)-\rho_1),\]
where $(\rho_1,\rho_2,\eta_1,\eta_2)$ are parameters that allow us to vary the ratio of magnitude of density and viscosity. The source term $\bef$ is computed accordingly. We note that a source term $f_\phi$ is also added to the level set equation so the level set is solution of the transport problem \eqref{eq:level_set_equation}.
We perform two set of tests using $(\rho_1,\rho_2)=(1,100), (\eta_1,\eta_2)=(1,10)$ and $(\rho_1,\rho_2)=(1,100), (\eta_1,\eta_2)=(1,0.01)$ for mesh sizes varying from $0.1$ to $0.00625$ with a time step defined by $\tau=h/2$. Our results are reported in Tables \ref{tab:semi_linear_eta}-\ref{tab:semi_linear_etab}. They recover an order of convergence equal to one under classic CFL condition. We also note that these results compare very well with a project method introduced by one of the authors that requires a Poisson problem to update the pressure, see Table 5 from \cite{cappanera_vu_2024}.
\begin{table}[ht] 
\caption{Semi-implicit scheme. Relative $L^2$ error for linear $\eta(\rho)$ with $(\rho_1,\rho_2)=(1,100), (\eta_1,\eta_2)=(1,10)$, $\tau=h/2$, $\nu_h=0.125 h$, $c_{comp}=0$, and $\lambda=1$.}
    \centering
    \begin{tabular}{|l|c|c||c|c||c|c|}
    	\hline
    	$\tau$ & \multicolumn{2}{c||}{Velocity}   & \multicolumn{2}{c||}{Pressure} & \multicolumn{2}{c|}{Density}  \\ \hline
    	& $L^2$ error & Order & $L^2$ error & Order& $L^2$ error & Order \\ \hline 
		0.05 	& 3.49E-2	&- 		& 2.80E0	&- 	& 1.03E-2	&- \\\hline
		0.025 	& 1.65E-2	& 1.09	& 1.16E0 	& 1.28	& 6.00E-3  	& 0.79 \\\hline
		0.0125 	& 8.14E-3	& 1.02	& 4.32E-1	& 1.42	& 3.17E-3	& 0.92 \\\hline
		0.00625 & 3.56E-3	& 1.19	& 1.08E-1	& 2.00	& 1.43E-3	& 1.15 \\\hline
		0.003125& 1.47E-3	& 1.28	& 4.74E-2	& 1.19	& 7.38E-04	& 0.95 \\
		\hline
    \end{tabular}
\label{tab:semi_linear_eta}
\end{table} 
\begin{table}[ht] 
\caption{Semi-implicit scheme. Relative $L^2$ error for linear $\eta(\rho)$ with $(\rho_1,\rho_2)=(1,100), (\eta_1,\eta_2)=(0.01,1)$, $\tau=h/2$, $\nu_h=0.125 h$, $c_{comp}=0$, and $\lambda=1$.}
    \centering
    \begin{tabular}{|l|c|c||c|c||c|c|}
    	\hline
    	$\tau$ & \multicolumn{2}{c||}{Velocity}   & \multicolumn{2}{c||}{Pressure} & \multicolumn{2}{c|}{Density}  \\ \hline
    	& $L^2$ error & Order & $L^2$ error & Order& $L^2$ error & Order \\ \hline 
		0.05 	& 5.95E-2  &- 	& 	2.39E0&- 	& 	1.05E-2&- \\\hline
		0.025 	& 2.37E-2  & 	1.33&  	4.72E-1& 	2.34&   5.98E-3& 0.82 \\\hline
		0.0125 	& 	7.99E-3& 	1.57& 	3.08E-1& 	0.62& 	3.21E-3&  0.90\\\hline
		0.00625 & 	3.01E-3& 	1.41& 	1.59E-1& 	0.95& 	1.46E-3&1.13 \\\hline
		0.003125& 	1.35E-3& 	1.16& 	9.61E-2& 0.73   & 	7.59E-4&  0.95\\
		\hline
    \end{tabular}
\label{tab:semi_linear_etab}
\end{table}

\subsubsection{Test with $\eta(\rho)$ nonlinear.}
For the second test, following \cite{cappanera_vu_2024}, we use the same setup that in the previous test to the difference that the dynamical viscosity is defined as nonlinear Lipschitz function of the density. We set:
\[\eta(x,y,t) = \eta(\rho) = \frac{1}{\rho(x,y,t)},\]
and update the definition of the source term $\bef$ accordingly. The parameter $(\rho_1,\rho_2)$ are set to $(1,100)$. Thus, in addition to consider a problem with large ratio of magnitude for the density, the kinematic viscosity $\nu=\eta/\rho$ has its magnitude varying from $1$ to $10^{-4}$. Thus, it presents an effective ratio of $10^4$ (versus $10^1$ in the previous section). Our results, that use the same space-time discretization than in the previous section, are reported in Table \ref{tab:semi_lipschitz_eta}. They confirm that the algorithm converges with order 1 under classic CFL condition even for problems with a nonlinear dynamical viscosity and a kinematic viscosity that presents large magnitude variation. Therefore, we infer that some of the hypothesis used in the stability and convergence analysis, see \eqref{eq:theorem1_hypo1}-\eqref{eq:theorem1_hypo2}, are heuristic and that in practice the algorithm is first-order and stable under the condition $\tau \propto h$. 
\begin{table}[ht] 
\caption{Semi-implicit scheme. Relative $L^2$ error for nonlinear $\eta(\rho)=\rho^{-1}$ with $(\rho_1,\rho_2)=(1,100)$, $\tau=h/2$, $\nu_h=0.125 h$, $c_{comp}=0$, and $\lambda=1$.}
    \centering
    \begin{tabular}{|l|c|c||c|c||c|c|}
    	\hline
    	$\tau$ & \multicolumn{2}{c||}{Velocity}   & \multicolumn{2}{c||}{Pressure} & \multicolumn{2}{c|}{Density}  \\ \hline
    	& $L^2$ error & Order & $L^2$ error & Order& $L^2$ error & Order \\ \hline 
		0.05 	& 5.19E-2&- 	& 2.86E0 &- 	& 1.03E-2	&- \\\hline
		0.025 	& 2.72E-2& 0.93	& 1.05E0 & 1.45	& 5.93E-3	& 0.80\\\hline
		0.0125 	& 1.54E-2& 0.82	& 3.38E-1& 1.64& 3.08E-3	& 0.94\\\hline
		0.00625 & 5.95E-3& 1.37	& 2.37E-1& 0.51	& 1.39E-3	& 1.15 \\\hline
		0.003125& 2.86E-3& 1.06	& 1.17E-1& 1.02	&	7.80E-4& 0.84\\
		\hline
    \end{tabular}
\label{tab:semi_lipschitz_eta}
\end{table}

\begin{remark} Although we only report three sets of tests, we note that the proposed artificial compressibility method has been validated over all the numerical setups considered in \cite{cappanera_vu_2024}. For all tests, our results showed that the error in velocity and level set were equivalent while the error in pressure was smaller for the projection method by a factor varying between 2 and 4. Thus, at this point, we can claim that the proposed artificial compressibility behaves similarly than the projection method described in \cite{cappanera_2018,cappanera_vu_2024} while not requiring to solve a Poisson problem to compute the pressure. The next section aims to show the advantages of using the artificial compressibility technique when studying immiscible incompressible multiphase flows.
\end{remark}

\subsection{Tests for explicit scheme with pseudo-spectral method}
\label{sec:num_SFEMaNS_results}
In this section, we investigate the stability and convergence properties of an explicit version of the scheme \eqref{eq:final_scheme_level_set}-\eqref{eq:final_scheme_u_and_p}. The main advantage of the explicit scheme is that it is suitable for spectral or pseudo-spectral code as nonlinearities are treated explicitly. Thus, all the numerical simulations reported in this section are performed with the pseudo-spectral code SFEMaNS. We remind that this 3D code uses a Fourier decomposition in the azimuthal direction and finite element in a 2D meridian plane $(r,z$). We refer to \cite{GLLN07} for more information on this software.

The objectives of this section are two folds.  First, we will show that the explicit counterpart of the semi-implicit scheme analyzed in sections \ref{sec:stab_algo}-\ref{sec:time_error_analysis} remains convergent and stable under classic CFL (i.e. $\tau\propto h$). Second, we will show that the proposed artificial compressibility method can approximate the solution of complex immiscible incompressible flows that the momentum-based projection method we previously introduced in \cite{cappanera_2018,cappanera_vu_2024} cannot approximate accurately.

\subsubsection{Time discretization for explicit algorithm}
\label{sec:num_explitcit_scheme}
The explicit counter part of the semi-implicit scheme \eqref{eq:final_scheme_level_set}-\eqref{eq:final_scheme_u_and_p} introduces two modifications. First, the transport term $\bu \cdot \GRAD \phi$, in the mass conservation equation, is treated explicitly. Second, the nonlinear term $(\bu \cdot\GRAD) \bmom$ in the Navier-Stokes equations is also treated explicitly. After initialization, the algorithm reads as follows for $n\geq 0$.

Step 1: Find $\phi^{n+1}$ solution of
\begin{equation}
\label{eq:scheme_exp_level_set}
\frac{\phi^{n+1} - \phi^n}{\tau}
+ \bu^n \cdot \GRAD \phi^n 
- \DIV (\nu_h \GRAD \phi^{n+1}) 
- \int_\Omega \left(c_{comp} \nu_h h^{-1} \phi^n (1-\phi^n) 
\frac{\GRAD \phi^n}{||\GRAD \phi^n||_{l^2}} \right) \GRAD \psi \diff\bx
=0,
\end{equation}
where the artificial viscosity $\nu_h$ is defined in section \ref{sec:num_level_set}.

Step 2: Compute the density $\rho^{n+1}$ using its linear dependency on $\phi^{n+1}$.

Step 3: Compute the dynamical viscosity $\eta^{n+1}$ using its dependency on $\rho^{n+1}$.

Step 4: Compute $\bmom^{n+1}$
\begin{multline}
\label{eq:scheme_exp_momentum}
\frac{\bmom^{n+1} - \bmom^n}{\tau}
- \DIV \left( \bar{\nu} \varepsilon(\bmom^{n+1} - \bmom^{*,n}) \right)
- \bar{\lambda} \GRAD \left( \DIV(\bmom^{n+1}-\bmom^{*,n}) \right)
= \bef^{n+1} 
- \DIV \left(\eta^{n+1} \varepsilon(\bu^{n}) \right)
- \GRAD p^n 
- \lambda \GRAD \left( \DIV \bu^n \right)
- (\bu^n \cdot \GRAD) \bmom^n,
\end{multline}
where $\bmom^{*,n}=\rho^{n+1}\bu^n$.

Step 5: Compute $\bu^{n+1}$ and $p^{n+1}$ using \eqref{eq:final_scheme_u_and_p}.

\begin{remark}
We note that the main advantages of the explicit algorithm, compared to the semi-implicit algorithm, is that the algorithm is suitable for pseudo-spectral methods. Moreover, it involves a stiffness matrix that is time independent so it can be assembled and preconditioned at initialization.
\end{remark}

\subsubsection{Smooth manufactured solution}
\label{sec:num_SFEMaNS_smooth_nonaxisym}

Using cylindrical coordinates,$(r,\theta,z)$, we define the computational domain as the cylinder $\Dom =\{ (r,\theta,z) \in [0,1]\times[0,2\pi) \times [-1,1] \}$. The time interval is set to $[0,1]$ and the analytical solutions are defined as follows
\begin{align*}
\phi(r,\theta,z,t) &= 0.5 
+ 0.25 \left( \cos(4(t-z)) 
+ r^2 \cos(4(t-z))  \sin(2 \theta) \right)
, \\
 \bu(r,\theta,z,t) &= \left(r z^2 \sin(t-z) \cos(2 \theta) , - r z^2 \sin(t-z) \sin(2 \theta), 1 \right)\tr
, \\ 
p(r,\theta,z,t) &= r^2 z^3 \cos(t) + r^2 \cos(t-z) \sin(\theta) + r z \sin(t-r) \cos(2\theta).
\end{align*}
The density and dynamical viscosity are reconstructed as a linear function of the level set by setting:
$$ \rho =\rho_0 + (\rho_1 - \rho_0) \phi,$$
$$\eta(\rho) = \eta_0 + (\eta_1 - \eta_0) \phi,$$
where the constant $(\rho_0, \rho_1)$, $(\eta_0,\eta_1)$ are user defined parameter used to set the maximum and minimum of density and dynamical viscosity. The source term $\bef$ is computed accordingly. 

We first check the convergence of our algorithm using a moderate ratio of density and viscosity, 
meaning that we set $(\rho_0,\rho_1)=(1,10)$, $(\eta_0,\eta_1)=(0.1,1)$. 
Then, we perform a second series of simulations with $(\rho_0,\rho_1)=(1,100)$, $(\eta_0,\eta_1)=(0.01,1)$ to check the influence of larger ratio of density and viscosities (dynamical and kinematic). 
The time step $\tau$ is set to $h/40$ with $h$ being the mesh size of the computational domain. 
All tests are performed with $c_\text{comp}=0$ and $\lambda=1$.
Our results are displayed in Tables \ref{tab:smooth_3D_rho_1_10_eta_0p1_1}-\ref{tab:smooth_3D_rho_1_100_eta_0p01_1}.
They show that the explicit version of the semi-implicit algorithm \eqref{eq:final_scheme_level_set}-\eqref{eq:final_scheme_u_and_p} is also first-order and stable under classic CFL condition when smooth level set functions are considered.
\begin{table}[ht]
\caption{Explicit scheme. Relative $L^2$ error for test with smooth level set, 
$\tau=h/40$, $\nu_h=0.125h$, $c_{comp}=0$, and $\lambda=1$. Degree of freedom per $\polP_2$ scalar unknown is denoted by $n_\text{df}$.}
\centering
\begin{tabular}{|c|c|c|c|c|c|c|c|c|} \hline\multicolumn{3}{|c|}{$\bL^2$-norm of error}
& \multicolumn{2}{|c|}{Velocity}
& \multicolumn{2}{|c|}{Pressure}
& \multicolumn{2}{|c|}{Level set} \\ \hline \hline 
&  $h$  & $n_\text{df}$ & Error  & Rate  & Error  & Rate   &Error & Rate   \\ \hline
\multirow{5}{*}{\shortstack{$(\rho_0,\rho_1)=(1,10)$ \\ $(\eta_0,\eta_1)=(0.1,1)$}}
&$0.1$    & 1017  		&1.23E-3 
 & -  	&1.37E-2 & - 	& 1.67E-2 & -   \\ \cline{2-9} 
&$0.05$  & 3821 & 6.61E-4 &0.90 & 	7.24E-3& 0.92&  8.96E-3&0.90 \\  \cline{2-9} 
&$0.02$  & 23437  &2.73E-4 & 1.28& 2.96E-3& 1.29 &3.71E-3 & 1.27\\  \cline{2-9} 
&$0.01$   &93173 & 1.38E-4&0.98 & 1.50E-3 & 	0.98& 1.88E-3& 0.98\\  \cline{2-9} 
&$0.005$ & 379689 &6.97E-5&0.99& 7.66E-4 &0.97 &9.54E-4 	& 0.98  \\
\hline
\end{tabular} 
\label{tab:smooth_3D_rho_1_10_eta_0p1_1}
\end{table}
\begin{table}[ht]
\caption{Explicit scheme. Relative $L^2$ error for test with smooth level set, 
$\tau=h/40$, $\nu_h=0.125h$, $c_{comp}=0$, and $\lambda=1$. Degree of freedom per $\polP_2$ scalar unknown is denoted by $n_\text{df}$. We set $\bmom^{*,n}=\rho^{n+1}\bu^n$ for mode {0}, and $\bmom^{*,n}=\rho^n\bu^n$ for all larger modes.}
\centering
\begin{tabular}{|c|c|c|c|c|c|c|c|c|} \hline\multicolumn{3}{|c|}{$\bL^2$-norm of error}
& \multicolumn{2}{|c|}{Velocity}
& \multicolumn{2}{|c|}{Pressure}
& \multicolumn{2}{|c|}{Level set} \\ \hline \hline 
&  $h$  & $n_\text{df}$ & Error  & Rate  & Error  & Rate   &Error & Rate   \\ \hline
\multirow{5}{*}{\shortstack{$(\rho_0,\rho_1)=(1,100)$ \\ $(\eta_0,\eta_1)=(0.01,1)$}}
&$0.1$    & 1017  		& 5.25E-3
 & -  	& 2.25E-1& - 	& 1.71E-2 & -   \\ \cline{2-9} 
&$0.05$  & 3821 &  2.63E-3& 0.99& 	1.15E-1& 0.97& 9.12E-3 &0.90 \\  \cline{2-9} 
&$0.02$  & 23437  &1.10E-3 & 1.26& 3.90E-2&  1.56& 3.77E-3& 1.28\\  \cline{2-9} 
&$0.01$   &93173 &5.56E-4 & 0.98&2.03E-2  & 0.94	& 1.92E-3& 0.98\\  \cline{2-9} 
&$0.005$ & 379689 &2.82E-4&0.98& 1.09E-2 &0.90 &9.70E-4 	&0.98   \\
\hline
\end{tabular} 
\label{tab:smooth_3D_rho_1_100_eta_0p01_1}
\end{table}

\subsubsection{Unsmooth axisymmetric manufactured solution}
\label{sec:num_SFEMaNS_nonsmooth_axisym}
After we checked the convergence properties of the algorithm for smooth level set in the previous section, we now show that the algorithm remains convergent and stable under classic CFL condition when the level set function, 
and so the density, is not smooth. 
This test is performed on the the same 
computational domain than in the previous section. 
The analytical solutions are defined by:
\[
\phi(r,\theta,z,t) = \mathbbm{1}_{(z<t-0.5)},
\qquad \bu(r,\theta,z,t) =  (0, r^2 \sin(t-z), 0.5)\tr,
\qquad p = r^2 z^3 \cos(t),
\]
where $\mathbbm{1}_{(z<t-0.5)}$ is the function that is equal to one when $z<t-0.5$ and is equal to zero elsewhere. The density is reconstructed as a linear function of the level set by setting $\rho=1+\phi$ while the dynamical viscosity is set to a constant $\eta=1$ so the source term does not involve a Dirac function. 
Therefore, for this setup, the source term $\bef$ is defined as follows:
\[
\bef=(
-\rho r^3 \sin^2(t-z) + 2r z^3 \cos(t), 
(r^2-3) \sin(t-z) + \frac{r^2 \rho}{2} \cos(t-z)
, 3 r^2 z^2 \cos(t)
)\tr.
\]
Since the level set presents a sharp profile, it can be seen as a test involving two immiscible fluids of respective density $\rho_0$ and $\rho_1$. Thus, all the test are performed using $c_\text{comp}=1$, see section \ref{sec:num_level_set} for more details on the compression technique. To run with a CFL number of order one (around $0.5$), the time step $\tau$ is set to $h/20$ where $h$ is the mesh size.
We perform a series of test with a mesh size $h$ ranging from $0.1$ to $0.005$ in $\polP_1$. Our results are summarized in Table \ref{tab:test_nonsmooth_3D}. We display the relative errors for the level set in $\bL^1$ norm and the relative errors for the velocity and pressure in $\bL^2$ norm.  We note that as the level set is discontinuous, the best rate of convergence in space is limited to one in $L^1$-norm and half in $L^2$-norm. Thus, we compute the level set error in $L^1$ norm, and not in $L^2$-norm like for the previous tests. Our results are consistent with a first-order scheme that is stable under classic CFL conditions. Thus, we conclude that the theoretical restrictive conditions\eqref{eq:theorem1_hypo1}-\eqref{eq:theorem1_hypo2} do not need to be enforced in practice and that we expect the semi-implicit and explicit algorithm to remain first order under the condition $\tau\propto h$. 
\begin{table}[ht]
\caption{Explicit scheme. Relative $L^2$ and $L^1$ errors for test with nonsmooth level set, 
$\tau=h/5$, $\nu_h=0.125h$, $c_{comp}=1$, and $\lambda=1$. Degree of freedom per $\polP_2$ scalar unknown is denoted by $n_\text{df}$.}
\centering
\begin{tabular}{|c|c|c|c|c|c|c|c|} \hline   
  \multicolumn{2}{|c|}{ }   & \multicolumn{2}{|c|}{Velocity, $\bL^2$-norm}
  & \multicolumn{2}{|c|}{Pressure, $L^2$-norm}
  & \multicolumn{2}{|c|}{Level set, $L^1$-norm} \\ \hline \hline 
	$h$		& $n_\text{df}$  & Error  & Rate & Error  & Rate  & Error  & Rate  \\ \hline
	$0.1$  	& 1017  & 1.52E-3  & -  	& 4.65E-3	& -  	& 3.01E-2 	& -  \\ \cline{1-8} 
	$0.05$ 	& 3821  & 7.89E-4 	& 0.95 	& 2.44E-3 	& 0.93 	& 1.62E-2	& 0.90 \\ \cline{1-8} 
	$0.02$  & 23437 & 3.22E-5 	& 0.98 	& 1.02E-3	& 0.96 	& 6.71E-3  & 0.96  \\  \cline{1-8} 
	$0.01$ 	& 93173 & 1.62E-5 	& 0.99 	& 5.21E-4  & 0.96 	& 3.35E-3	& 1.00 \\  \cline{1-8} 
	$0.005$	& 379689& 8.15E-5	& 1.00	& 2.71E-4 	& 0.94 	& 1.66E-3 	& 1.02  \\
\hline
\end{tabular}
\label{tab:test_nonsmooth_3D}
\end{table}

\subsubsection{Gravitational waves in cylinder}
\label{sec:num_MPR_NST}

We conclude our numerical investigations by presenting a more challenging setup where a stratification of two immiscible fluids is driven by gravitational waves. This setup was originally introduced in \cite{herreman2019perturbation} as a benchmark to validate magnetohydrodynamics software for two fluids and was extended to three fluids in \cite{herreman2023stability}. The original idea of the magnetohydrodynamics setup is to reproduce the metal pad roll instability. This instability is well-known in metallurgy as it leads to a tilting and rotating motion of the liquid metal in an aluminum production cell because of the combination of gravity and Lorentz forces. The hydrodynamics set up, considered here, focuses on studying a tilting motion of the fluids' interface due to gravitational waves. Thanks to a linear perturbation analysis, it is shown in \cite{herreman2019perturbation,herreman2023stability} that introducing a small perturbation of the initial interface at rest ($z=0$) allows us to determine theoretically the dispersion relation $\omega$ (related to the frequency of oscillations) and the viscous damping (related to the magnitude of the oscillations) of the instability. Moreover, both quantities only depend on the physical parameters of the problem (e.g. domain, fluids' density). In the following, we briefly describe the hydrodynamic setting considered here. Then, we discuss the results we obtain with the artificial compressibility technique described in section \ref{sec:num_explitcit_scheme} and compare them with the results obtained with the projection method analyzed in \cite{cappanera_vu_2024}.

The computational domain is set to the cylinder of radius two centimeters and height to four centimeters and we set
$D=\{(r,\theta,z)\in [0.02]\times[0,2\pi]\times[-0.02,0.02] \}$. The cylinder is filled with a stratification of two immiscible fluids of respective density $(\rho_1,\rho_2)$, and respective dynamical viscosity $(\eta_1,\eta_2)$ such that the heavier fluid lays at the bottom of the cylinder. At initialization, the interface between both fluids is set to $z=0$. The heights of the bottom and top fluids are denoted by $H_1, H_2$ and are both equal to $0.02m$. The initial interface is then perturbed by adding a velocity that creates a vertical deformation of the form $A J_m(k_{m1} r) e^{(im\theta+\omega t)}e^{\lambda t}$ where $A$ is an arbitrary amplitude, $J_m$ is the first kind Bessel function, $K=k_{m1}R^{-1}$ where $R=0.02$ is the radius of the cylinder and $k_{m1}$ if the first zero of $J_m^\prime$. The dispersion relation is then given by:
\begin{equation}
    \omega=\pm \sqrt{\frac{(\rho_2-\rho_1)gK}{\rho_1 \tanh^{-1}{(KH_1)}+\rho_2 \tanh^{-1}{(KH_2)}}},
\end{equation}
where the gravity is set to $g=9.81ms^{-2}$.
The viscous damping $\lambda_{visc}$ is given in \cite{herreman2019perturbation}, see formula 2.66. The effective dispersion relation is given by $\omega_{corr} = \omega + \lambda_{visc}$. We refer to \cite{herreman2019perturbation} for more details on the physical setup.

In the following, we set $A=10^{-6}$, $m=1$, $H_1=H_2=R=0.02$. The fluids properties are set to:
$$
\begin{cases}
(\rho_1, \rho_2)=(6077,13475) kg.m^{-3},\\
(\eta_1, \eta_2)=(1.686,1.417)\times 10^{-3} m^{2}.s^{-1}.\end{cases}
$$
Therefore, the theoretical dispersion relations and viscous damping are given by:
$$
\begin{cases}
\omega=18.0259865641184, \\
\lambda_{visc}=-0.106761084127129,\\
\omega_{corr}=\omega+\lambda_{visc}=17.9192254799912.
\end{cases}\\
$$
We run two series of tests where the mesh size is set respectively to $h=0.001$ and $h=0.0005$ in $\polP_1$. The final time is set to $T=5s$ and each series of tests is performed with time steps $\tau$ varying from $0.02$ to $0.00025$. The results we obtain with the explicit artificial compressibility technique are displayed in Table \ref{tab:test_MPR_NST_hmult3}. All simulations are performed with $\lambda=1$ and $c_{comp}=1$. They confirm that the algorithm is first-order. On the other hand, we also display the results we obtain with the projection method from \cite{cappanera_vu_2024} in Table \ref{tab:test_MPR_NST_hmult3_proj_sym}. It shows that the projection method, that also uses the momentum as primary unknown, fails to capture the correct dynamics. This phenomena was originally observed by one of the authors when studying incompressible multiphase magnetohydrodynamics setup whose fluids's interface undergoes strong deformation \cite{nore2021feasibility,herreman2021efficient}. It led us to develop the artificial compressibility technique that we presented and analyzed in this paper.


\begin{table}[ht]
\caption{Gravitational wave: errors obtained with explicit artificial compressibility techniques. $\nu_h=0.125 h$, $c_{comp}=1$, and $\lambda=1$.}
\centering
\begin{tabular}{|c|c|c|c|c|c|c|c|} \hline
$h$ &  $\tau$  & $\omega_{num}$ 
	& Error $\omega_{corr} $ & Rate
	& $\lambda_{visc-num}$ & Error $\lambda_{visc}$  & Rate \\ \hline \hline
\multirow{4}{*}{0.001}
& $0.02$    & 1.53E1 & 1.47E-1 & --& -2.95E-1 & 1.77E0 & -- \\ \cline{2-8} 
& $0.001$   & 1.67E1 & 6.99E-2 & 1.07& -2.24E-1 & 1.10E0 & 0.69 \\ \cline{2-8} 
& $0.0005$  & 1.74E1 & 2.79E-2 & 1.32& -1.69E-1 & 5.86E-1 & 0.91 \\ \cline{2-8} 
& $0.00025$ & 1.78E1 & 6.29E-3 & 2.15& -1.39E-1 & 3.06E-1 & 0.94 \\ \hline \hline
\multirow{4}{*}{0.0005}
& $0.02$    & 1.53E1 & 1.43E-1 & -- & -2.70E-01 & 1.53E0 & -- \\ \cline{2-8} 
& $0.001$   & 1.68E1 & 6.37E-2 & 1.17& -2.21E-1 & 1.07E0 & 0.51 \\ \cline{2-8}
& $0.0005$  & 1.76E1 & 1.95E-2 &1.71 & -1.64E-1 & 5.38E-1 & 0.99 \\ \cline{2-8} 
& $0.00025$ & 1.80E1 & 4.56E-3 &2.47 & -1.34E-1 & 2.55E-1 & 1.08 \\ \hline
\end{tabular} 
\label{tab:test_MPR_NST_hmult3}
\end{table}

\begin{table}[ht]
\caption{Gravitational waves: errors obtained with explicit projection method from \cite{cappanera_vu_2024}. $\nu_h=0.125 h$ and $c_{comp}=1$.}
\centering
\begin{tabular}{|c|c|c|c|c|c|c|c|} \hline
$h$ &  $\tau$  & $\omega_{num}$ 
	& Error $\omega_{corr} $ & Rate
	& $\lambda_{visc-num}$ 
    & Error $\lambda_{visc}$  & Rate 
    \\ \hline \hline
\multirow{4}{*}{0.001}
& $0.02$    & 1.62E1 & 9.59E-2 & -- & -5.69E-1 & 4.33E0 & -- \\ \cline{2-8} 
& $0.001$   & 1.63E1 & 8.95E-2 & 0.10  & -4.51E-1  & 3.22E0 & 0.43 \\ \cline{2-8}
& $0.0005$  & 1.63E1 & 8.76E-2 & 0.03  & -3.87E-1  & 2.62E0 & 0.30 \\ \cline{2-8} 
& $0.00025$ & 1.63E1 & 8.72E-2 & 0.01  & -3.54E-1  & 2.32E0 & 0.18 \\ \hline
\hline
\multirow{4}{*}{0.0005}
& $0.02$    & 1.63E1 & 8.99E-2 & -- & -6.07E-1 & 4.69E0 & -- \\ \cline{2-8} 
& $0.001$   & 1.65E1 & 7.71E-2 & 0.22    & -4.94E-1 & 3.63E0 & 0.37 \\ \cline{2-8}
& $0.0005$  & 1.66E1 & 7.36E-2 & 0.07    & -4.30E-1 & 3.03E0 & 0.26 \\ \cline{2-8} 
& $0.00025$ & 1.66E1 & 7.25E-2 & 0.02    & -3.97E-1 & 2.72E0 & 0.16 \\ \hline
\end{tabular} 
\label{tab:test_MPR_NST_hmult3_proj_sym}
\end{table}


\section{Conclusions}\label{sec:conclusion}
We introduce a semi-implicit algorithm for the incompressible Navier-Stokes 
equations with variable density and viscosity that combines
an artificial compressibility method and a momentum based approximation. We establish the stability and temporal convergence of the proposed algorithm under the assumption that the density approximation satisfies a min-max principle and is first order in time. While theoretical results involve non optimal criterion, like a time step bounded by the square of the density's gradient magnitude, we showed that in practice the algorithm is stable and convergent under classic CFL condition even when the density presents large ratio of magnitude or is discontinuous. Eventually, we also introduce an explicit version of the algorithm and show that it maintains the same stability and convergence properties than its semi-implicit counterpart but has the advantage of being suitable for spectral methods. To compare our artificial compressibility method with a projection-type algorithm introduced by one of the author in \cite{cappanera_2018}, we consider similar setups than the one studied in \cite{cappanera_vu_2024} which shows that both methods behave similarly on problems with manufactured solutions. We conclude our investigations with a more challenging setup, involving immiscible incompressible flow driven by gravitational waves. Our results show that the proposed method recovers theoretical results, such as wave's oscillation frequency and viscous dumping from \cite{herreman2019perturbation} while the projection method from \cite{cappanera_2018} does not.
For future investigations, we consider developing a second-order method based on the recent development of high-order artificial compressibility methods for flows with constant density in \cite{guermond2017high,guermond2019high} and flows with variable density in \cite{lundgren2023high}. 
Another direction consists in adapting the Scalar Auxiliary Variable (SAV) technique, see \cite{shen2018convergence,jiang2023artificial,zhang2023mass}, to our framework which would result in an unconditionally stable explicit artificial compressibility method for flows with variable density.





\bmsection*{Financial disclosure}

This work was supported by the National Science Foundation NSF (L. Cappanera, grant number DMS-2208046)

\bmsection*{Conflict of interest}

The authors declare that they have no potential conflict of interest.

\bibliography{wileyNJD-AMS}

\begin{thebibliography}{10}
\providecommand{\url}[1]{\texttt{#1}}
\providecommand{\urlprefix}{URL }
\expandafter\ifx\csname urlstyle\endcsname\relax
  \providecommand{\doi}[1]{doi:\discretionary{}{}{}#1}\else
  \providecommand{\doi}{doi:\discretionary{}{}{}\begingroup
  \urlstyle{rm}\Url}\fi

\bibitem{anderson1998diffuse}
D.~M. Anderson, G.~B. McFadden, and A.~A. Wheeler, \textit{Diffuse-interface
  methods in fluid mechanics}, Annual review of fluid mechanics \textbf{30}
  (1998), no.~1, 139--165.

\bibitem{badalassi2003computation}
V.~Badalassi, H.~Ceniceros, and S.~Banerjee, \textit{Computation of multiphase
  systems with phase field models}, Journal of computational physics
  \textbf{190} (2003), no.~2, 371--397.

\bibitem{badia2017monotonicity}
S.~Badia and J.~Bonilla, \textit{Monotonicity-preserving finite element schemes
  based on differentiable nonlinear stabilization}, Computer Methods in Applied
  Mechanics and Engineering \textbf{313} (2017), 133--158.

\bibitem{barrenechea2017algebraic}
G.~R. Barrenechea, V.~John, and P.~Knobloch, \textit{An algebraic flux
  correction scheme satisfying the discrete maximum principle and linearity
  preservation on general meshes}, Mathematical Models and Methods in Applied
  Sciences \textbf{27} (2017), no.~03, 525--548.

\bibitem{benzi2006augmented}
M.~Benzi and M.~A. Olshanskii, \textit{An augmented lagrangian-based approach
  to the oseen problem}, SIAM Journal on Scientific Computing \textbf{28}
  (2006), no.~6, 2095--2113.

\bibitem{boris1973flux}
J.~P. Boris and D.~L. Book, \textit{Flux-corrected transport. i. shasta, a
  fluid transport algorithm that works}, Journal of computational physics
  \textbf{11} (1973), no.~1, 38--69.

\bibitem{cahouet1988some}
J.~Cahouet and J.-P. Chabard, \textit{Some fast 3d finite element solvers for
  the generalized stokes problem}, International Journal for Numerical Methods
  in Fluids \textbf{8} (1988), no.~8, 869--895.

\bibitem{cappanera_vu_2024}
L.~Cappanera and A.~Vu, \textit{Stability and error analysis of a semi-implicit
  scheme for incompressible flows with variable density and viscosity}, Journal
  of Numerical Mathematics  (2024).

\bibitem{cappanera_2018}
L.~Cappanera, J.-L. Guermond, W.~Herreman, and C.~Nore, \textit{Momentum-based
  approximation of incompressible multiphase fluid flows}, International
  Journal for Numerical Methods in Fluids \textbf{86} (2018), no.~8, 541--563.

\bibitem{chiu2011conservative}
P.-H. Chiu and Y.-T. Lin, \textit{A conservative phase field method for solving
  incompressible two-phase flows}, Journal of Computational Physics
  \textbf{230} (2011), no.~1, 185--204.

\bibitem{Chor68}
A.~Chorin, \textit{Numerical solution of the {N}avier-{S}tokes equations},
  Math. Comp. \textbf{22} (1968), 745--762.

\bibitem{chorin1967numerical}
A.~J. Chorin, \textit{A numerical method for solving incompressible viscous
  flow problems}, Journal of computational physics \textbf{2} (1967), no.~1,
  12--26.

\bibitem{COUPEZ_NUMIFORM2007}
T.~Coupez, \textit{{C}onvection {O}f {L}ocal {L}evel {S}et {F}unction {F}or
  {M}oving {S}urfaces {A}nd {I}nterfaces {I}n {F}orming {F}low}, J.~C. de~S\'a
  and A.~Santos (eds.), \textit{AIP Conference Proceedings}, vol. 908(1), 2007.

\bibitem{COUPEZ2010}
T.~Coupez, L.~Silva, and L.~Ville, \textit{Convected level set method for the
  numerical simulation of fluid buckling}, Internat. J. Numer. Methods Fluids
  \textbf{66} (2010), 324--344.

\bibitem{dong2012time}
S.~Dong and J.~Shen, \textit{A time-stepping scheme involving constant
  coefficient matrices for phase-field simulations of two-phase incompressible
  flows with large density ratios}, Journal of Computational Physics
  \textbf{231} (2012), no.~17, 5788--5804.

\bibitem{fortin1983augmented}
M.~Fortin and R.~Glowinski, \textit{Augmented lagrangian methods, volume 15 of
  studies in mathematics and its applications}  (1983).

\bibitem{gottlieb1977numerical}
D.~Gottlieb and S.~A. Orszag, \textit{Numerical analysis of spectral methods:
  theory and applications}, vol.~26, Siam, 1977.

\bibitem{guermond2017conservative}
J.-L. Guermond, M.~Q. de~Luna, and T.~Thompson, \textit{An conservative
  anti-diffusion technique for the level set method}, Journal of Computational
  and Applied Mathematics \textbf{321} (2017), 448--468.

\bibitem{guermond2015high}
J.-L. Guermond and P.~Minev, \textit{High-order time stepping for the
  incompressible navier--stokes equations}, SIAM Journal on Scientific
  Computing \textbf{37} (2015), no.~6, A2656--A2681.

\bibitem{guermond2019high}
J.-L. Guermond and P.~Minev, \textit{High-order adaptive time stepping for the
  incompressible navier--stokes equations}, SIAM Journal on Scientific
  Computing \textbf{41} (2019), no.~2, A770--A788.

\bibitem{MR2250931}
J.~L. Guermond, P.~Minev, and J.~Shen, \textit{An overview of projection
  methods for incompressible flows}, Comput. Methods Appl. Mech. Engrg.
  \textbf{195} (2006), no. 44-47, 6011--6045.

\bibitem{guermond2017high}
J.-L. Guermond and P.~D. Minev, \textit{High-order time stepping for the
  navier--stokes equations with minimal computational complexity}, Journal of
  Computational and Applied Mathematics \textbf{310} (2017), 92--103.

\bibitem{GLLN07}
J.-L. Guermond, R.~Laguerre, J.~L{\'e}orat, and C.~Nore, \textit{An interior
  penalty {G}alerkin method for the {MHD} equations in heterogeneous domains},
  J. Comput. Phys. \textbf{221} (2007), no.~1, 349--369.

\bibitem{GLLN09}
J.-L. Guermond, R.~Laguerre, J.~L{\'e}orat, and C.~Nore, \textit{Nonlinear
  magnetohydrodynamics in axisymmetric heterogeneous domains using a
  {Fourier}/finite element technique and an interior penalty method}, J.
  Comput. Phys. \textbf{228} (2009), 2739--2757.

\bibitem{HARTENIII}
A.~Harten, \textit{The artificial compression method for computation of shocks
  and contact discontinuities. {III}. {S}elf-adjusting hybrid schemes}, Math.
  Comp. \textbf{32} (1978), 363--389.

\bibitem{harten1984class}
A.~Harten, \textit{On a class of high resolution total-variation-stable
  finite-difference schemes}, SIAM Journal on Numerical Analysis \textbf{21}
  (1984), no.~1, 1--23.

\bibitem{harten1997high}
A.~Harten, \textit{High resolution schemes for hyperbolic conservation laws},
  Journal of computational physics \textbf{135} (1997), no.~2, 260--278.

\bibitem{hecht2012new}
F.~Hecht, \textit{New development in freefem++}, Journal of numerical
  mathematics \textbf{20} (2012), no. 3-4, 251--266.

\bibitem{herreman2019perturbation}
W.~Herreman, C.~Nore, J.-L. Guermond, L.~Cappanera, N.~Weber, and G.~Horstmann,
  \textit{Perturbation theory for metal pad roll instability in cylindrical
  reduction cells}, Journal of Fluid Mechanics \textbf{878} (2019), 598--646.

\bibitem{herreman2021efficient}
W.~Herreman, C.~Nore, L.~Cappanera, and J.-L. Guermond, \textit{Efficient
  mixing by swirling electrovortex flows in liquid metal batteries}, Journal of
  Fluid Mechanics \textbf{915} (2021), A17.

\bibitem{herreman2023stability}
W.~Herreman, L.~Wierzchalek, G.~Horstmann, L.~Cappanera, and C.~Nore,
  \textit{Stability theory for metal pad roll in cylindrical liquid metal
  batteries}, Journal of Fluid Mechanics \textbf{962} (2023), A6.

\bibitem{jiang2023artificial}
N.~Jiang and H.~Yang, \textit{Artificial compressibility sav ensemble
  algorithms for the incompressible navier-stokes equations}, Numerical
  Algorithms \textbf{92} (2023), no.~4, 2161--2188.

\bibitem{kuzmin2012algebraic}
D.~Kuzmin, \textit{Algebraic flux correction I: Scalar conservation laws},
  Springer, 2012.

\bibitem{kuzmin2020monolithic}
D.~Kuzmin, \textit{Monolithic convex limiting for continuous finite element
  discretizations of hyperbolic conservation laws}, Computer Methods in Applied
  Mechanics and Engineering \textbf{361} (2020), 112804.

\bibitem{kuzmin2012flux}
D.~Kuzmin, R.~L{\"o}hner, and S.~Turek, \textit{Flux-corrected transport:
  principles, algorithms, and applications}, Springer, 2012.

\bibitem{ladyzhenskaya1969mathematical}
O.~A. Ladyzhenskaya, \textit{The mathematical theory of viscous incompressible
  flow}, vol.~2, Gordon and Breach New York, 1969.

\bibitem{li2021}
M.~Li, Y.~Cheng, J.~Shen, and X.~Zhang, \textit{A bound-preserving high order
  scheme for variable density incompressible navier-stokes equations}, Journal
  of computational physics \textbf{425} (2021), 109906.

\bibitem{liu2003phase}
C.~Liu and J.~Shen, \textit{A phase field model for the mixture of two
  incompressible fluids and its approximation by a fourier-spectral method},
  Physica D: Nonlinear Phenomena \textbf{179} (2003), no. 3-4, 211--228.

\bibitem{lundgren2023high}
L.~Lundgren and M.~Nazarov, \textit{A high-order artificial compressibility
  method based on taylor series time-stepping for variable density flow},
  Journal of Computational and Applied Mathematics \textbf{421} (2023), 114846.

\bibitem{manzanero2020entropy}
J.~Manzanero, G.~Rubio, D.~A. Kopriva, E.~Ferrer, and E.~Valero, \textit{An
  entropy--stable discontinuous galerkin approximation for the incompressible
  navier--stokes equations with variable density and artificial
  compressibility}, Journal of Computational Physics \textbf{408} (2020),
  109241.

\bibitem{nazarov2013residual}
M.~Nazarov and J.~Hoffman, \textit{Residual-based artificial viscosity for
  simulation of turbulent compressible flow using adaptive finite element
  methods}, International Journal for Numerical Methods in Fluids \textbf{71}
  (2013), no.~3, 339--357.

\bibitem{nore2021feasibility}
C.~Nore, L.~Cappanera, J.-L. Guermond, T.~Weier, and W.~Herreman,
  \textit{Feasibility of metal pad roll instability experiments at room
  temperature}, Physical Review Letters \textbf{126} (2021), no.~18, 184501.

\bibitem{olshanskii2022recycling}
M.~A. Olshanskii and A.~Zhiliakov, \textit{Recycling augmented lagrangian
  preconditioner in an incompressible fluid solver}, Numerical Linear Algebra
  with Applications \textbf{29} (2022), no.~2, e2415.

\bibitem{Kreiss}
E.~Olsson and G.~Kreiss, \textit{A conservative level set method for two phase
  flow}, J. Comput. Phys. \textbf{210} (2005), 225--246.

\bibitem{osher1993level}
S.~Osher, \textit{A level set formulation for the solution of the dirichlet
  problem for hamilton--jacobi equations}, SIAM Journal on Mathematical
  Analysis \textbf{24} (1993), no.~5, 1145--1152.

\bibitem{osher2001level}
S.~Osher and R.~P. Fedkiw, \textit{Level set methods: an overview and some
  recent results}, Journal of Computational physics \textbf{169} (2001), no.~2,
  463--502.

\bibitem{osher1988fronts}
S.~Osher and J.~A. Sethian, \textit{Fronts propagating with curvature-dependent
  speed: algorithms based on hamilton-jacobi formulations}, Journal of
  computational physics \textbf{79} (1988), no.~1, 12--49.

\bibitem{shen1995error}
J.~Shen, \textit{On error estimates of the penalty method for unsteady
  navier--stokes equations}, SIAM Journal on Numerical Analysis \textbf{32}
  (1995), no.~2, 386--403.

\bibitem{shen2018convergence}
J.~Shen and J.~Xu, \textit{Convergence and error analysis for the scalar
  auxiliary variable (sav) schemes to gradient flows}, SIAM Journal on
  Numerical Analysis \textbf{56} (2018), no.~5, 2895--2912.

\bibitem{stiernstrom2021residual}
V.~Stiernstr{\"o}m, L.~Lundgren, M.~Nazarov, and K.~Mattsson, \textit{A
  residual-based artificial viscosity finite difference method for scalar
  conservation laws}, Journal of Computational Physics \textbf{430} (2021),
  110100.

\bibitem{temam1968methode}
R.~Temam, \textit{Une m{\'e}thode d'approximation de la solution des
  {\'e}quations de navier-stokes}, Bulletin de la Soci{\'e}t{\'e}
  Math{\'e}matique de France \textbf{96} (1968), 115--152.

\bibitem{tema69}
R.~Temam, \textit{Sur l'approximation de la solution des \'equations de
  {N}avier-{S}tokes par la m\'ethode des pas fractionnaires {\rm ii}}, Arch.
  Rat. Mech. Anal. \textbf{33} (1969), 377--385.

\bibitem{zalesak1979fully}
S.~T. Zalesak, \textit{Fully multidimensional flux-corrected transport
  algorithms for fluids}, Journal of computational physics \textbf{31} (1979),
  no.~3, 335--362.

\bibitem{zhang2023mass}
Y.~Zhang, H.~Dong, and K.~Wang, \textit{Mass, momentum and energy
  identical-relation-preserving scheme for the navier-stokes equations with
  variable density}, Computers \& Mathematics with Applications \textbf{137}
  (2023), 73--92.

\end{thebibliography}

\end{document}